\documentclass[10.5pt]{amsart}
\usepackage[ansinew]{inputenc}
\usepackage{amssymb}
\usepackage{latexsym}
\usepackage{graphics,color}
\usepackage[dvips]{graphicx}
\usepackage{amsmath,amsfonts}
\usepackage{tikz}
\usepackage{enumitem}

\newtheorem{theorem}{Theorem}[section]
\newtheorem{lemma}[theorem]{Lemma}
\newtheorem{prop}[theorem]{Proposition}

\newtheorem{ass}{Assumption}[section]
\newtheorem{corollary}[theorem]{Corollary}
\newtheorem{definition}[theorem]{Definition}

\usepackage{amssymb,amsfonts,amsmath,amsthm,color}
\usepackage{latexsym}
\usepackage{graphics}
\usepackage{amsfonts}
\usepackage{amsbsy}
\usepackage{amsmath}
\usepackage{indentfirst}
\usepackage{graphicx}
\usepackage{geometry}
\usepackage{dsfont}

    \def\D{\mathcal{D}}
    \def\E{\mathcal{E}}
        \def\F{\mathcal{F}}
    \def\G{\mathcal{G}}
    \def\h{\mathbb{H}^1_{\Gamma}}
    \def\H{\mathcal{H}}

    \def\L{\mathbb{L}^2}
    
    \def\M{\mathcal{M}}
    \def\N{\mathbb{N}}
    
    \def\R{\mathbb{R}}

\def\<{\langle}
\def\>{\rangle}

\renewcommand{\d}{\partial}

\def\beginpf{\noindent {\bf Proof:} \quad}
\def\endpf{\rightline{$\square$}}

\def\restriction#1#2{\mathchoice
              {\setbox1\hbox{${\displaystyle #1}_{\scriptstyle #2}$}
              \restrictionaux{#1}{#2}}
              {\setbox1\hbox{${\textstyle #1}_{\scriptstyle #2}$}
              \restrictionaux{#1}{#2}}
              {\setbox1\hbox{${\scriptstyle #1}_{\scriptscriptstyle #2}$}
              \restrictionaux{#1}{#2}}
              {\setbox1\hbox{${\scriptscriptstyle #1}_{\scriptscriptstyle #2}$}
              \restrictionaux{#1}{#2}}}
\def\restrictionaux#1#2{{#1\,\smash{\vrule height .8\ht1 depth .85\dp1}}_{\,#2}}
\begin{document}

\author{R\'emi Buffe $^1$} 
\footnotetext[1]{Inria, Villers-l\`es-Nancy, F-54600, France. E-mail : \texttt{remi.buffe@inria.fr}}
	
\author{ Marcelo M. Cavalcanti $^2$}
\footnotetext[2]{ Department of Mathematics, State University of
	Maring\'a, 87020-900, Maring\'a, PR, Brazil. E-mail : \texttt{mmcavalcanti@uem.br}}

\author{ Val\'eria N. Domingos Cavalcanti $^3$ }
\footnotetext[3]{ Department of Mathematics, State University of
	Maring\'a, 87020-900, Maring\'a, PR, Brazil. E-mail : \texttt{vndcavalcanti@uem.br}}

\author{ Ludovick Gagnon $^{4,*}$}
\footnotetext[4]{Inria, Villers-l\`es-Nancy, F-54600, France. E-mail : \texttt{ludovick.gagnon@inria.fr},  \textit{corresponding author$^*$}}

\thanks{Research of Marcelo M. Cavalcanti is partially supported by the CNPq Grant
	300631/2003-0. Research of Val\'eria N. Domingos Cavalcanti is partially supported by the CNPq Grant
	304895/2003-2. Research of Ludovick Gagnon is partially supported by ANR ODISSE and by the R\'eseau Franco-Br\'esilien en Math\'ematiques.}

\title[Transmission Problem of Viscoelastic Waves]
{Control and Exponential Stability for a Transmission Problem of a Viscoelastic Wave Equation}

\maketitle

\begin{abstract}

In this article, we consider the energy decay of a viscoelastic wave in an heterogeneous medium. To be more specific, the medium is composed of two different homogeneous medium with a memory term located in one of the medium. We prove exponential decay of the energy of the solution under geometrical and analytical hypothesis on the memory term. 
\textup{2020} \it Mathematics Subject Classification: 35Q93, 35A27, 35L05, 35L51, 35S15.

\end{abstract}

\section{Introduction}

\subsection{Description of the problem}

Let $\Omega \subset \R^n$ be an open and bounded domain. Let $\Omega_2 \subset \Omega$ such that $\overline{\Omega_2} \subset \Omega$. Define $\Omega_1=\Omega \setminus \overline{\Omega_2}$. The boundary of $\Omega_1$ is therefore given by $\d \Omega_1= \d \Omega \cup \d \Omega_2$. We assume  $\d \Omega $ and $ \d \Omega_2$ to be of class $C^k, k\geq 3$ and with no contact of order $k-1$ with its tangents. The outward unit normal of $\Omega_2$ is denoted $n$.  

We consider a nonnegative function $b(x)\in C^{\infty}(\Omega_1)$. We are interested in studying the {\bf exponential stability} of 
\begin{equation}\label{MP1}\qquad
\left\{
\begin{aligned}
& u_{tt} -k_1\Delta u + k_1\int_{-\infty}^t g(t-s) \textrm{ div}(b(x) \nabla u)(s)\,ds =0, & ~\hbox{ in }\Omega_1 \times (0,\infty),\\
& v_{tt} - k_2\Delta v=0, & ~\hbox{ in }\Omega_2 \times (0,\infty),
\end{aligned}
\right.
\end{equation}
 the boundary condition
\begin{equation}\label{bc}
u=0~\hbox{ on } \d \Omega \times (0,\infty),
\end{equation}
and the transmission conditions
\begin{equation}\label{fronteira}
u=v ~\mbox{ and  } k_2\dfrac{\partial v}{\partial n} = k_1\dfrac{\partial u}{\partial n}-k_1 \int_{-\infty}^t g(t-s)b(x) \dfrac{\partial u}{\partial n}(s)\,ds~\hbox{ on }\d \Omega_2\times (0,\infty),
\end{equation}
where $k_1$ and $k_2$ are positive constants, each one related to the propagation velocity of waves in media $\Omega_1$ and $\Omega_2$, respectively. 

\begin{figure}[!ht]
\begin{center}
	\includegraphics[height=4.5cm]{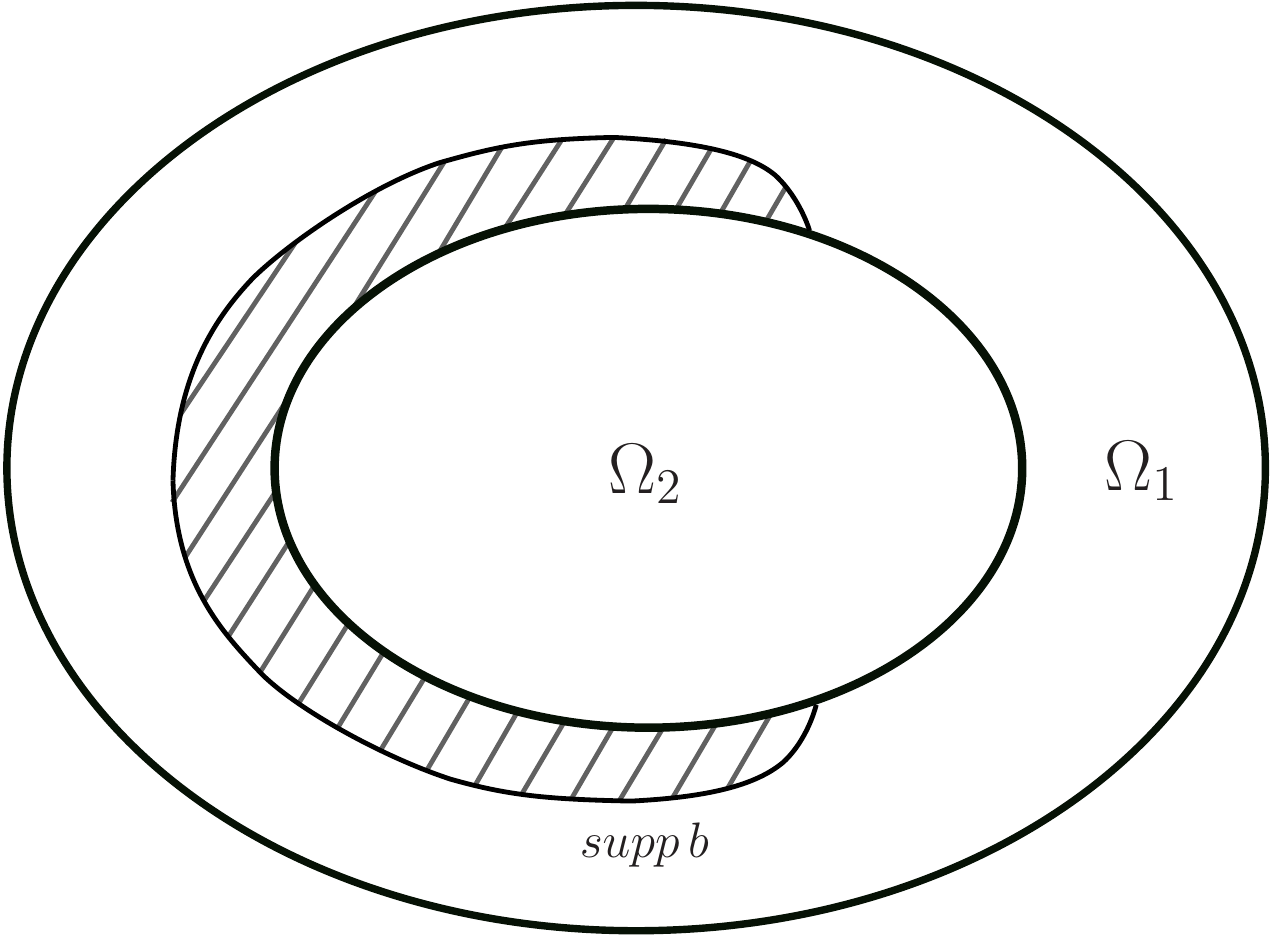} \quad
	\includegraphics[height=4.5cm]{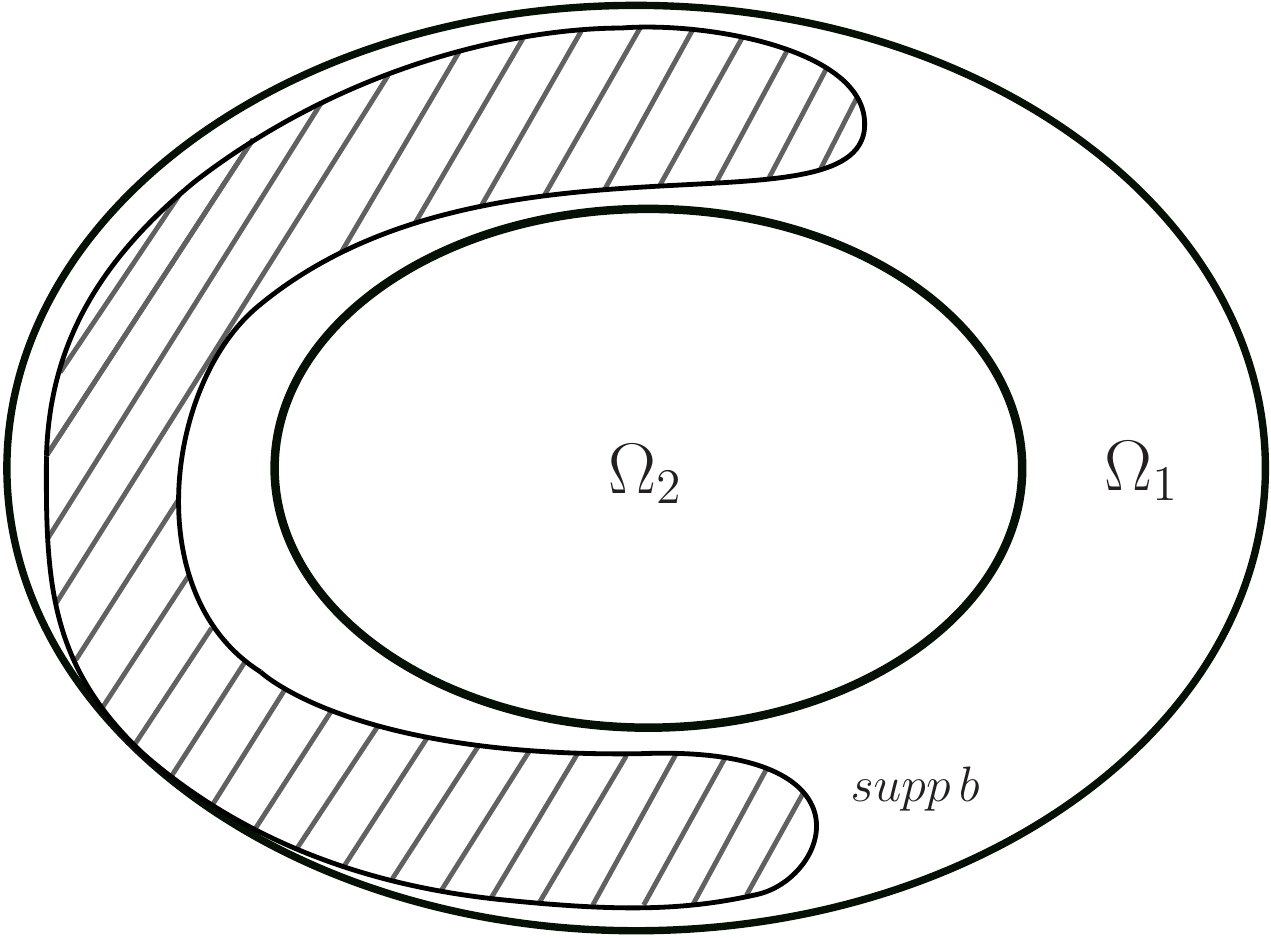}
\end{center}
\caption{\footnotesize{Representation of the spatial domain for \eqref{MP1} in the case where $supp(b)$ touches the interface (left) and the case where it does not (right).}}
\end{figure}

The function $u$ satisfies  on $\Omega_1 \times (0,\infty)$ the equation $(\ref{MP1})_1$ and on $\Omega_1 \times (-\infty,0]$ verifies
\begin{equation}\label{pasthistory}
u(x,-t)=\phi_0 (x,t), \quad (x, t) \in \Omega_1 \times [0,\infty),
\end{equation}
where $\phi_0: \Omega_1 \times  [0,\infty)\rightarrow \mathbb{R}$ is the prescribed past history of $u$.

In addition, $(u,v)$ satisfy the initial data
\begin{equation}\label{id}\begin{array}{c}u(x,0)=u_0(x);~u_t(x,0)=u_1(x), x\in \Omega_1,\\
v(x,0)=v_0(x); v_t(x,0)=v_1(x), x\in \Omega_2.
\end{array}\end{equation}

Note that equation $(\ref{MP1})_1$ can be written as 

$$u_{tt} -k_1\Delta u + k_1\int_0^{\infty}g(s) \textrm{ div} (b(x) \nabla u)(t-s)\,ds =0,$$
or, equivalently,
\begin{equation}\label{equationvisco}
u_{tt} -k_1\textrm{ div}\left[ \nabla u(t) -\int_0^{\infty}g(s)  b(x) \nabla u(t-s)\,ds\right] =0.
\end{equation}

Defining $G(s)= 1- \int_0^s g(\xi) \, d\xi$, the above equation turns out to be 
$$u_{tt} -k_1\textrm{ div}\left[ G(0)\nabla u(x,t) -\int_0^{\infty} G'(s)  b(x) \nabla u(x,t-s)\,ds\right] =0, ~ (x,t) \in\Omega_1 \times (0,\infty),$$
which describes the evolution of the displacement field $u$ in a homogeneous isotropic solid, whose viscoelastic part is localized in the support of $b=b(x)$, $\hbox{supp }b \subset \Omega_1$, and it occupies $\Omega_1$ at rest, see for example Fabrizio et al. \cite{Fabrizio} and the references therein.  

Using the past history framework, introduced by Dafermos in his pioneering paper \cite{Dafermos}, it was possible to treat equation (\ref{equationvisco}) in a different way.

Introducing the change of variables 
\begin{eqnarray}\label{change}
\eta(x,t,s) := u(x,t) - u(x,t-s),
\end{eqnarray}
 we deduce
%with 
%\[
%u(x,-t)=\tilde{u}(x,-t), \quad (x, t) \in \Omega_1 \times (0,\infty),
%\]
%where $\tilde{u}: \Omega_1 \times (-\infty ,0)\rightarrow \mathbb{R}$ is the prescribed past history of $u$.

%From this change of variables, we deduce 
\begin{align*}
\int_{-\infty}^t g(t-s) \textrm{ div}(b(x)\nabla u(x,s)) ds & = \int_{0}^\infty g(s) \textrm{ div}(b(x)\nabla u(x,t-s)) ds \\
& = \left(\int_{0}^\infty g(s) ds \right) \textrm{ div}(b(x)\nabla u)(x,t)  - \int_{0}^\infty g(s)  \textrm{ div}(b(x)\nabla \eta(x,t,s)) ds. 
\end{align*}
Likewise, the flux condition writes 
\begin{align*}
\int_{-\infty}^t g(t-s) b(x) \dfrac{\d u}{\d n}(x,s) ds & =  \int_{0}^\infty g(s) b(x) \dfrac{\d u}{\d n}(x,t-s) ds \\
& = \left(\int_{0}^\infty g(s) ds \right) b(x) \dfrac{\d u}{\d n}(x,t)  - \int_{0}^\infty g(s) b(x)  \dfrac{\d \eta}{\d n}(x,t,s) ds. 
\end{align*}
Defining
\[
k_0:=\int_{0}^\infty g(s) ds, 
\]
equation (\ref{equationvisco}) translates into the following system:
\begin{equation}\label{change}\qquad
\left\{
\begin{aligned}
& u_{tt} - k_1 \textrm{div }\left[(1-k_0 b(x)) \nabla u + \int_0^\infty g(s) b(x)\nabla \eta(s)\,ds\right]=0, \\
&\eta_t + \eta_s =u_t.
\end{aligned}
\right.
\end{equation}

and the non-autonomous problem  (\ref{MP1}) is transformed into the equivalent autonomous one
\begin{equation}\label{AUP}\qquad
\left\{
\begin{aligned}
& u_{tt} - k_1 \textrm{div }( (1-k_0 b(x)) \nabla u) - k_1 \int_0^\infty g(s) \textrm{div }(b(x)\nabla \eta(s))\,ds=0, & ~\hbox{ in }\Omega_1 \times (0,\infty),\\
&\eta_t + \eta_s =u_t, & ~\hbox{ in } \Omega_1 \times (0,\infty) \times (0,\infty),\\
&v_{tt} -k_2 \Delta v =0, & ~\hbox{ in }\Omega_2 \times (0,\infty),\\
& u=\eta=0, & ~\hbox{ on }\d \Omega\times(0,\infty)\\
& u=v, &  ~\hbox{ on }\d \Omega_2 \times (0,\infty)\\
& k_2 \dfrac{\partial v}{\partial \nu}= k_1 (1-k_0 b(x))\dfrac{\partial u}{\partial \nu}  +k_1 \int_0^\infty g(s)b(x)\dfrac{\partial \eta}{\partial \nu} (s)\,ds, & ~\hbox{ on }\d \Omega_2 \times (0,\infty),\\
& u(x,0)=u_0(x); ~u_t(x,0)=u_1(x), & ~x\in \Omega_1,\\
& v(x,0)=v_0(x); ~v_t(x,0)=v_1(x), & ~x\in \Omega_2,\\
&\eta(x,t,0)=0;~\eta(x,0,s)=\eta_0(x,s)=u_0-\phi_0(s), & ~x\in \Omega_1.
\end{aligned}
\right.
\end{equation}
Throughout this article, we assume that  $1-k_0 b(x) \leq 1, \forall x\in \overline{\Omega}_1$ and therefore $k_1(1-k_0 b(x)) \leq k_2, \forall x\in \overline{\Omega}_1$, whenever $k_1<k_2$. Moreover, 
$$
\eta(x,t,0):=\lim_{s\rightarrow 0_+}\eta(x,t,s)=0, \qquad (x,t)\in \Omega\times(0,\infty).$$

According to Theorem 2.1 and Theorem 2.2 in Dafermos \cite{Dafermos}, the past history function regularity , as well as the other initial data, implies the continuity of the solution regarding the time parameter, in the interval $(-\infty, T]$, for $T>0$.

%It should be noted that the initial data of the problem are fixed so that the solution retains the continuity properties of the solution at time $t=0$, that is, $(u,v) \in C((-\infty,\infty);X)$ with $X$ to be determined later. The context of equation \eqref{AUP} is that the memory term acts in the past and that the observation is initiated at time $t=0$.
%
%\begin{equation*}
%\left\{
%\begin{aligned}
%&u_0(x)=u_0(x,0), &~x\in \Omega_1,\\
%&u_1(x)= \partial_t u_0(x,t)|_{t=0}, & ~x\in \Omega_1,\\
%&\eta_0(x,s)=u_0(x,0) - u_0(x,-s), & ~x\in \Omega_1.
%\end{aligned}
%\right.
%\end{equation*}

Observe that the energy functional associated to problem (\ref{MP1}) is defined by
\begin{eqnarray*}
E(t)&:=&\frac12\int_{\Omega_1}\left[|u_t(x,t)|^2 + k_1\, |\nabla u(x,t)|^2\right]\,dx\\
&+& \frac12\int_{\Omega_2}\left[|v_t(x,t)|^2 + k_2\, |\nabla v(x,t)|^2\right]\,dx\nonumber
\end{eqnarray*}
and, under this form, the energy decay of \eqref{MP1} is not easy to establish since the sign of 
\begin{eqnarray*}
E'(t) = k_1\int_{-\infty}^t g(t-s) \int_\Omega b(x) \nabla u(x,s)\cdot \nabla u_t(x,t) \,dx ds,~t\geq 0,
\end{eqnarray*}
is difficult to control. Moreover, \eqref{MP1} being non-autonomous, arguing by an observability inequality argument is unlikely to yield the exponential decay of the solution.

On the other hand, using the boundary conditions, the second line of \eqref{AUP} and that $g(s)\rightarrow 0$ as $s\rightarrow \infty$, one obtains that the energy  functional of (\ref{AUP}) defined by
\begin{eqnarray}\label{energy}
E(t)&:=&\frac12\int_{\Omega_1}\left[|u_t(x,t)|^2 + (1-k_0 b(x)) \, |\nabla u(x,t)|^2\right]\,dx + \frac12\int_{\Omega_2}\left[|v_t(x,t)|^2 + k_2\, |\nabla v(x,t)|^2\right]\,dx\nonumber \\
&+& k_1 \int_0^\infty g(s)\int_{\Omega_1} b(x)|\nabla \eta(x,t,s)|^2\,dxds, 
\end{eqnarray}
satisfies the identity 
\begin{eqnarray}\label{energy-identity}
E(t_2) - E(t_1) = \frac{k_1}2 \int_{t_1}^{t_2}\int_0^{+\infty} g'(s)\, \int_{\Omega_1} b(x)|\nabla \eta(x,t,s) |^2\,dxdsdt,
\end{eqnarray}
for all $0\leq t_1 < t_2$. Further imposing that $g$ is decreasing implies that \eqref{energy-identity} is nonpositive for any $0\leq t_1 < t_2$. 

In order to achieve the exponential stability, we are going to establish that it is sufficient to prove that for all $T>T_0,$  there exists a constant $C>0$ such that the following observability inequality holds:
\begin{eqnarray}\label{Obs ineq}\,\,\qquad
E(0)\leq C
\left(\int_0^T\int_{0}^{\infty}(-g'(s)) \int_{\Omega_1}  b(x)|\nabla \eta(x,t,s)|^2 \, dxdsdt\right).
\end{eqnarray}

The exponential stability result is the main goal of the present paper and will be established in section 3 under geometrical assumptions. In the next section we are going to present the notations, the analytical hypothesis and the functional spaces which will be used throughout the paper as well as the well-posedness.

\subsection{Literature overview and main contribution of the present article}

There are several articles in connection with the controllability and stabilization of wave transmission problems.
Initially, we would like to mention some important papers related to the exact controllability of transmission problems associated with the wave equations. The question of boundary controllability in problems of transmission has been considered by several authors. In particular Lions \cite{LionsBook1} considered the system in the special case of two wave equations, namely (using Lions' notations),
\begin{eqnarray*}
	\begin{cases}
		\partial_t^2 y_1 - a_1 \Delta y_1 =0 \quad \hbox{ in }\Omega_1 \times (0,T),\\
		\partial_t^2 y_2 - a_2 \Delta y_2 =0 \quad \hbox{ in }\Omega_2 \times (0,T),
	\end{cases}
\end{eqnarray*}
where $\Omega, \Omega_1$ are bounded open connected sets in $\mathbb{R}^n$ with smooth boundaries $\Gamma$ and $\Gamma_1$ respectively such that $\overline{\Omega_1} \subset \Omega$ and $\Omega_2:= \Omega\backslash \Omega_1$ whose boundary is $\Gamma_2 := \Gamma \cup \Gamma_1$. Here, $a_i>0$ ($i=1,2$) and $\Delta$ is the ordinary Laplacian in $\mathbb{R}^n$,
\begin{eqnarray*}
	&&y_2 = v \hbox{ on }\Sigma=\Gamma \times (0,T), v \hbox{ is the control},\\
	&& y_1=y_2,~a_1 \partial_\nu y_1= a_2 \partial_\nu y_2\hbox{ on }\partial \Omega_i,~i=1,2,\\
	&& y_i|_{t=0}= \partial_t y_i|_{t=0}=0 \hbox{ on }\Omega_i,~i=1,2.
\end{eqnarray*}

Assuming that $\Omega_1$ is star shaped with respect to some point $x_0\in \Gamma_1$ and setting $\Gamma(x_0):=\{x\in \Gamma:(x-x_0)\cdot\nu(x)>0\}$, $\Sigma(x_0):=\Gamma(x_0) \times (0,T)$ where $\nu$ is the unit outer normal to $\Gamma$, Lions proved the exact boundary controllability assuming that $a_1>a_2$ and for $T>T(x_0)=2R(x_0)/\sqrt{a_2}$ and $R(x_0)=\max_{x\in \overline{\Omega_2}}|x-x_0|$.

Later on Lagnese \cite{26} generalized Lions \cite{LionsBook1} by considering transmission problems for general second order linear hyperbolic systems having piecewise constant coefficients in a bounded, open connected set with smooth boundary and controlled through the Dirichlet boundary condition. It is showed that such a system is exactly controllable in an appropriate function space provided the interfaces where the coefficients have a jump discontinuity are all star shaped with respect to one and the same point and the coefficients satisfy a certain monotonicity condition.

Another interesting generalization of Lions \cite{LionsBook1} has been considered by Liu \cite{27}. In this paper the author addresses the problem of control of the transmission wave equation. In particular, he considers the case where, due to total internal reflection of waves at the interface, the system may not be controlled from exterior boundaries. He shows that such a system can be controlled by introducing both boundary control along the exterior boundary and distributed control near the transmission boundary and give a physical explanation why the additional control near the transmission boundary might be needed for some domains.

We also would like to quote the papers due to Nicaise \cite{28}, \cite{29} in which the author discusses the problem of exact controllability of networks of elastic polygonal membranes. The individual membranes are assumed to be coupled either at a vertex or along a whole common edge.  The author then derives energy estimates for regular solutions, which are then, by transposition, extended to weak solutions. As usual, direct and inverse inequalities of the type shown in these articles establish a norm equivalence on a certain space (classically named $F$), the completion of which is the space in which the HUM-principle of Lions works. The space $F'$ then contains the null-controllable initial data. This space is weak enough to correspond to $L^2$-boundary controls along exterior edges satisfying sign conditions with respect to energy multipliers, to such controls along Dirichlet-edges, and, more importantly, to $H^1$-vertex controls at those vertices which are responsible for severe singularities. The corresponding solutions, for $(u_0,u_1)\in H\times V'$ with rather weak regularity $(C(0,T,D(A)'))$, are then shown to be null-controllable in a canonical finite time.

Another very nice paper that we would like to quote is the work of Miller \cite{Miller}, which although not related to controllability is very closed to the subject of investigation . This article deals with the propagation of high-frequency wave solutions to the scalar wave and  Schr\"odinger equations. The results are formulated in terms of semiclassical measures (Wigner measures). The propagation is across a sharp interface between two inhomogeneous media. The author proves a microlocal version of Snell-Descartes's law of refraction which includes diffractive rays. Moreover, a radiation phenomenon for density of waves propagating inside an interface along gliding rays is illustrated. The measures of the traces of the solutions of the corresponding partial differential equations enable the author to derive some propagation properties for the measure of the solutions.   

Finally we would like to mention the recent papers due to Gagnon \cite{LG} and Astudillo et al \cite{Astudillo}. In the first one \cite{LG} the author considers waves traveling in two different mediums each endowed with a different constant speed of propagation. At the interface between the two mediums, the refraction of the rays of the optic geometry
is described by the Snell-Descartes's law. The author introduced a geometrical construction that yields sufficient geometrical conditions under the hypothesis that $\Omega$ and $\Omega_2$ and that the boundary observability region is given by $\Gamma=\Gamma(x_0)$. More precisely, the geometrical construction allows to keep track of the propagation of the generalized bicharacteristics as they encounter the interface. This is the critical issue since, from interference phenomenon at the interface, concentration of energy on outgoing rays from the interface is possible, as, for every ray incoming at the interface, there exists a reflected ray, a transmitted ray and an interference ray if the Snell-Descartes's law is not vacuous. In particular, the classical \textit{Geometrical Control Condition} is not appropriate in this setting. Roughly speaking, the geometrical construction found in \cite{LG}, using microlocal defect measure, uses an iterative process to propagate the observability region $\Gamma(x_0)$ to a subset $\Gamma_2 \subset \d \Omega_2$ of the interface that is observable, meaning that every ray encountering transversally $\Gamma_2$ is observed. Moreover, it is shown that, under the geometrical assumptions, $\Gamma_2$ satisfies GCC for $\Omega_2$. Once the iterative geometrical construction is over, one is left with a part of the domain $\Omega_1^f \subset \Omega_1$ in which one is not able to conclude on the observability of the generalized bicharacteristics propagating in $\Omega_1^f$. The sufficient conditions is then expressed in terms of a uniform escaping geometry condition on $\Omega_1^f$, requiring every rays in $\Omega_1^f$, as well as the transmitted rays in $\Omega_2$, to propagate "directly" outside $\Omega_1^f$ and toward the observability region. This condition is similar to asking that there are no trapped rays in $\Omega_1^f$, but is more subtle in the sense that one has to consider $\Omega_2$ not as an obstacle but a region where transmission occurs. Therefore, one also has to make sure that the rays propagating in $\Omega_2$ propagate uniformly toward the observability region as well to prevent the interference phenomenon to occur. The geometrical proof of the present paper uses this construction to derive sufficient geometrical conditions for distributed controls from boundary controls. As we shall see, the geometrical conditions relies on the geometrical construction for the boundary controls as well as the additional non trapping condition for the rays propagating only in $\Omega_1$. 

In the second aforementioned one  \cite{Astudillo} the authors  study the exact boundary controllability of a generalized wave equation in a non smooth domain with a nontrapping obstacle. In the more general case, this work contemplates the boundary control of a transmission problem admitting several zones of transmission. The result is obtained using the technique developed by David Russell, taking advantage of the local energy decay for the problem, obtained through the Scattering Theory as used by Vodev \cite{23, 33}, combined with a powerful trace Theorem due to Tataru \cite{4}.

On the other hand, in what concerns the stabilization of wave transmission problems associated to an internal  frictional dissipation, we would to quoted the following paper \cite{Cavalcanti} due to Cavalcanti et al and references therein. In this paper, the authors obtain very general decay rate estimates associated to a wave-wave transmission problem subject to a nonlinear damping locally distributed and they present explicit decay rate estimates of the associated energy. In addition, they implement a precise and efficient code to study the behavior of the transmission problem when $k_1\ne k_2$ and when one has a nonlinear frictional dissipation $g(u_t )$. More precisely, they aim to numerically check the general decay rate estimates of the energy associated to the problem established in first part of the paper. It is worth mentioning the paper of Cardoso and Vodev \cite{23} . In this paper the authors  study the transmission problem in bounded domains with dissipative boundary conditions. Under some natural assumptions, they prove uniform bounds of the corresponding resolvents on the real axis at high frequency and, as a consequence, they obtain regions free of eigenvalue.  As an application,
the authors get exponential decay of the energy of the solutions of the corresponding mixed boundary value problems.

Regarding the stabilization of wave transmission problems associated to a viscoelastic effects, as far as we are concerned, the unique paper of the literature published so far is the following one \cite{AMO}. This paper goes in the same direction of the present one with two drawbacks: (i) $b(x) =1$  which implies that a full damping is in place in $\Omega_1$, ~
(ii) In \cite{AMO} the authors consider an additional frictional damping acting on a collar of the transmission zone which characterizes an over damping. From the above, the main contribution of the present article is to generalize substantially the aforementioned article by removing the excess of viscoelastic and frictional dissipations. For this purpose refined arguments of micro local analysis are taken into account.

\section{Assumptions, Functional Spaces and The Well-Posedness Result}

\begin{ass}\label{ass:1}

 $\mathbf{(i)}$ $b\in C^\infty(\Omega_1) \cap C^0 (\overline{\Omega}_1)$ is a nonnegative function 
such that there is a positive distance between $\hbox{supp }b$ and $\d \Omega_2$.\\	
$\mathbf{(ii)}$ $  g\in L^1([0,\infty))\cap C^1([0,\infty)) $  is a positive non-increasing function satisfying
\begin{eqnarray}\label{solid}
	l:= 1-k_0||b||_{L^{\infty}(\Omega_1)} >0,
	\end{eqnarray}
for some positive constant $C>0$, where	
\begin{equation}\label{solid0}
\int_{0}^{\infty} g(s) \, ds :=k_0.\\
\end{equation}
\end{ass}

Let us define
$$
\mathbb{H}^1 = H^1(\Omega_1) \times H^1(\Omega_2);~\mathbb{L}^2= L^2(\Omega_1) \times L^2(\Omega_2),$$
$$\mathbb{H}_{\d \Omega}^1= \left\{(u,v)\in \mathbb{H}^1; u|_{\d \Omega}=0 \mbox{ and } u=v\hbox{ on }\d \Omega_2\right\},
$$
where $\mathbb{H}^1_{\d \Omega}$ is a Hilbert space endowed with the inner product
\begin{eqnarray*}
\left((u_1,v_1), (u_2,v_2) \right)_{\mathbb{H}_{\d \Omega}^1}= {k_1}\int_{\Omega_1} (1-k_0 b(x))\nabla u_1
\cdot \nabla {u_2}\,dx + k_2\int_{\Omega_2}\nabla v_1
\cdot \nabla {v_2}\,dx,
\end{eqnarray*}
which is equivalent to the usual norm of $H^1_0 (\Omega)$ taking (\ref{solid}) into account.

We shall introduce the notations
\begin{eqnarray*}
|| u||_1^2 = \int_{\Omega_1} |u(x)|^2\,dx ~\hbox{ and }||v||_2^2= \int_{\Omega_2} |v(x)|^2\,dx,
\end{eqnarray*}
where $(u,v)\in \mathbb{L}^2$ and
\begin{eqnarray*}
\left((u_1,v_1), (u_2,v_2) \right)_{\mathbb{L}^2}= (u_1,u_2)_1 +  (v_1,v_2)_2 = \int_{\Omega_1} u_1(x)u_2(x)\,dx
 +\int_{\Omega_2} v_1(x)v_2(x)\,dx,
\end{eqnarray*}
with $(u_1,v_1), (u_2,v_2)\in \mathbb{L}^2$.

We define the space
\begin{eqnarray*}
L_g^2( \mathbb{R}^+;V)=\left\{ \eta; \int_0^\infty g(s) ||\eta(s)||_V^2\,ds<+\infty\right\},
\end{eqnarray*}
where
\begin{eqnarray*}
V=\{w\in L^2(\Omega_1); \sqrt{b(x)}\nabla w \in L^2(\Omega_1), w=0\hbox{ on }\d \Omega\},
\end{eqnarray*}
which is a Hilbert space, endowed with the following inner product
\begin{eqnarray*}
(\eta,\xi)_{L^2_g}=k_1\int_0^\infty g(s) \int_{\Omega_1} b(x) \nabla \eta \cdot \nabla {\xi}\,dxds + k_1\int_0^\infty g(s) \int_{\Omega_1} \eta {\xi}\,dxds  ,~\hbox{ for all }\eta,\xi\in L_g^2( \mathbb{R}^+;V).
\end{eqnarray*}

The hypothesis imposed on function $b$ are fundamental to prove that $V$ is well defined and, in addition, is a Hilbert space.  As a consequence, it makes sense to consider the trace of order zero of any function $u$ belonging to this space.

Furthermore, we consider
$$D(T)= \{\eta; \eta , \eta_s \in L_g^2( \mathbb{R}^+;V),\eta(0)=0\}$$
and the operator \begin{equation*}\begin{array}{rccc}T :& D(T)\subset L_g^2( \mathbb{R}^+;V)& \longmapsto&L_g^2( \mathbb{R}^+;V)\\& \eta& \rightarrow& T(\eta) = -\eta_s.\end{array}\end{equation*}

Finally, we introduce the state space
$$ X = \h\times\L\times L_g^2( \mathbb{R}^+;V).$$

Defining the linear operator
\begin{equation}\label{operator} \mathcal{A}:D(\mathcal{A})\subset X \longrightarrow X \end{equation}
$$ \mathcal{A} U =\left[\begin{array}{c}u_2\\v_2\\ k_1 \textrm{div }( (1-k_0 b(x)) \nabla u_1) + k_1 \int_0^\infty g(s) \textrm{div }(b(x)\nabla \eta(s))\,ds \\k_2 \Delta v_1\\ u_2 - \eta_s\end{array}\right],$$
where $U=((u_1,v_1),(u_2,v_2),\eta)^T$ and
$$\begin{array}{r} \D( \mathcal{A}) =  \Big\{u=((u_1,v_1),(u_2,v_2),\eta)\in \mathcal{H};  (u_2,v_2) \in \mathbb{H}_{\d \Omega}^1, \eta \in D(T),\\ \left(\textrm{div }( (1-k_0 b(x)) \nabla u_1) + \int_0^\infty g(s) \textrm{div }(b(x)\nabla \eta(s))\,ds,\Delta v_1\right)\in \L \\ \mbox{ and }
k_2 \frac{\partial v_1}{\partial \nu} = k_1 (1-k_0 b(x))\dfrac{\partial u_1}{\partial \nu}  +k_1 \int_0^\infty g(s)b(x)\dfrac{\partial \eta}{\partial \nu} (s)\,ds,  ~\hbox{ on }\d \Omega_2 \Big\}\end{array},$$
 problem (\ref{AUP}) is equivalent to the Cauchy Problem
\begin{equation}\label{Cauchy-problem}
\left\{\begin{array}{l}
\displaystyle\frac{d}{dt} U (t) = \mathcal{A} U(t), \quad t>0, \medskip \\
U(0)=U_0,
\end{array}\right.
\end{equation}
where $ U_0=((u_0,v_0),(u_1,v_1),\eta_0).$

\subsection{Well-posedness result}

Now we are in a position to establish the well-posedness result for \eqref{Cauchy-problem}, which ensures that problem \eqref{AUP} is globally well posed.

\begin{theorem}[Global Well-posedness]
 	\label{theo-existence} Under {\bf Assumption \ref{ass:1}} we have
 	\begin{itemize}
 		\item[$(i)$]  If $U_0\in X $, then problem \eqref{Cauchy-problem}  has a unique mild
 		solution $U\in C([0,\infty),X)$.
 		
 		\medskip
 		
 		\item[$(ii)$]  If  $U_0\in D(\mathcal{A})$, then the above mild
 		solution $U$ is regular with
 		$$
 		U\in C([0,\infty),D({\mathcal{A}}))\cap  C^{1}([0,\infty),X).
 		$$
\end{itemize}
\end{theorem}

 \begin{proof}
	
Let $U=((u_1,v_1),(u_2,v_2),\eta)\in D({\mathcal{A}})$ and $\omega >0$ verifying $\omega>\max\left\{\dfrac{k_1 k_0}{2},\dfrac{1}{2}\right\}>0$, then

\begin{align*}
((\mathcal{A} &-\omega I) U,U)_{X} \\
&=k_1\int_{\Omega_1}(1-k_0b(x))\nabla u_1\nabla u_2 dx -k_1 \omega
\int_{\Omega_1}(1-k_0b(x))|\nabla u_1|^2dx\\
&+\int_{\Omega_1}\left\{k_1\hbox{div}[(1-k_0a(x))\nabla u_1]+k_1\int_0^\infty g(s)
\hbox{div}(b(x)\nabla \eta)ds\right\}u_2 dx  + k_2\int _{\Omega_2}  \Delta v_1 v_2 dx\\ 
&+ k_2\int_{\Omega_2}\nabla v_1
\cdot \nabla {v_2}\,dx -\omega k_2 \int_{\Omega_2} |\nabla v_1|^2dx - \omega \int_{\Omega_1} | u_2|^2dx- \omega \int_{\Omega_2} | v_2|^2dx\\
&-k_1 \int_0^\infty \int_{\Omega_1}g(s)\eta_s\eta
dxds- k_1 \int_0^\infty \int_{\Omega_1}g(s){b(x)}\nabla \eta_s\nabla \eta
dxds+ k_1 \int_0^\infty \int_{\Omega_1}g(s)u_2 \eta dxds
\\
&+ k_1 \int_0^\infty \int_{\Omega_1}g(s){b(x)}\nabla u_2\nabla \eta dxds - k_1 \omega \int_0^\infty \int_{\Omega_1}g(s)|\eta|^2
dxds {- k_1\omega \int_0^\infty \int_{\Omega_1}g(s)b(x)|\nabla
	\eta|^2dxds}.
\end{align*}

Considering that for every element $U=((u_1,v_1),(u_2,v_2),\eta)\in D({\mathcal{A}})$  we have $u_1= u_2 =0$ on $\d \Omega$, $u_1=v_1$, $u_2 =v_2$ and $ k_2 \dfrac{\partial v_1}{\partial \nu}= k_1 (1-k_0 b(x))\dfrac{\partial u_1}{\partial \nu}  +k_1 \int_0^\infty g(s)b(x)\dfrac{\partial \eta}{\partial \nu} (s)\,ds$ on $\d \Omega_2$, it yields

\begin{equation}\label{dissipative}
((\mathcal{A} -\omega I) U,U)_{X} \\
\leq -(\omega -\dfrac{k_1 k_0}{2}) \int_{\Omega_1} | u_2|^2dx - k_1(\omega - \dfrac{1}{2}) \int_0^\infty \int_{\Omega_1}g(s)|\eta|^2
dxds \leq 0.
\end{equation}

Defining $\mathcal{B}:=\mathcal{A} -\omega I $, in light of (\ref{dissipative}), we conclude that $\mathcal{B}$ is a dissipative operator, that is, $-\mathcal{B}$ is a monotone operator.  Now, we need to prove that  $R(I-\mathcal{B})=X$, or equivalently,
$R((1+\omega)I-\mathcal{A})=X.$

\noindent Indeed, given  $(f_1, f_2, f_3)\in X$ we shall prove that there exists $((u_1,v_1),(u_2,v_2),\eta)\in D(\mathcal{A})$ such that
\begin{equation}
\left\{\begin{array}{l} (1+\omega)(u_1,v_1)-(u_2,v_2)=f_1:= (f_{11}, f_{12}),\label{3.2}\\
(1+\omega)(u_2,v_2)- k_1(\{\hbox{div}[(1-k_0b(x))\nabla u_1]+k_1 \int_0^\infty g(s) \hbox{div}
[b(x)\nabla\eta]ds\}, k_2 \Delta v_1) =f_2 := (f_{21}, f_{22}),\\
(1+\omega)\eta+\eta_s- u_2=f_3.
\end{array} \right.
\end{equation}

Using equation $(\ref{3.2})_3$ we, formally, obtain
\begin{equation}\label{3.5}
\eta(s)=\int_0^s
f_3(\tau)e^{(1+\omega)(\tau-s)}d\tau+\dfrac{u_2}{1+\omega}(1-e^{-(1+\omega)s})\end{equation}

and observing equation $ (\ref{3.2})_1$ we conclude that
\begin{equation}\label{3.6}
\left\{\begin{array}{l} u_1=\dfrac{1} {1+\omega}(f_{11} +u_2)\\
 v_1=\dfrac{1} {1+\omega}(f_{12} +v_2).
\end{array} \right.
\end{equation}

Replacing (\ref{3.5}) and (\ref{3.6}) in $(\ref{3.2})_2$ it derives

\begin{align}\label{u2}
(1+\omega)u_2& -\dfrac{k_1}{1+\omega}\hbox{div}[(1-k_0b(x))+ c^* b(x)]\nabla
u_2\\ \nonumber
&=f_{21}+\dfrac{k_1}{1+\omega}\hbox{div}[(1-k_0b(x))\nabla
f_{11}]+k_1\int_0^\infty g(s) \int_0^s
e^{(1+\omega)(\tau-s)} \hbox{div}[b(x)\nabla f_3(\tau)]d\tau
ds,
\end{align}
where $ c^* = \int_0^\infty g(s)(1-e^{-(1+\omega)s})ds$, and

\begin{eqnarray}\label{v2}
(1+\omega)v_2-\dfrac{k_2}{1+\omega}\Delta v_2 = f_{22}+\dfrac{k_2}{1+\omega} \Delta f_{12}.
\end{eqnarray}

The above two identities are the motivation to define the bilinear form, which is continuous and coercive on $\h\times \h$ 
\begin{equation}\label{bilform}
\begin{array}{lcl}
 B((z_1,w_1),(z_2,w_2))
&=&\displaystyle\int_{\Omega_1}(1+\omega) z_1 z_2\,dx+\dfrac{k_1}{1+\omega}\int_{\Omega_1}[(1-k_0b(x))+c^*b(x)] \nabla z_1 \cdot \nabla z_2\,dx\\
&=&\displaystyle\int_{\Omega_2}(1+\omega) w_1 w_2\,dx+\dfrac{k_2}{1+\omega}\int_{\Omega_2} \nabla w_1 \cdot \nabla w_2\,dx
\end{array}
\end{equation}
for all $ (z_1,w_1),(z_2,w_2)\in \h$; and also the motivation to define the following linear and continuous operator:
$$\begin{array}{rccc} \mathcal{L}: &\h & \longrightarrow &\R \\ &(w,z)&\mapsto& \langle\mathcal{L}, (w,z)\rangle \end{array}$$
given by
\begin{align}
 \langle\mathcal{L}, (w,z)\rangle &= \int_{\Omega_1} f_{21} w\, dx - \dfrac{k_1}{1+\omega}\int_{\Omega_1}(1-k_0b(x))\nabla f_{11} \cdot \nabla w \,dx\\
&- k_1\int_0^\infty g(s) \int_0^s e^{(1+\omega)(\tau-s)} \int_{\Omega_1}b(x)\nabla f_3(\tau) \cdot \nabla w \, dx d\tau ds\\
&+ \int_{\Omega_2} f_{22} z\, dx - \dfrac{k_2}{1+\omega}\int_{\Omega_2} \nabla f_{12} \cdot \nabla z \,dx.
\end{align}

Using the Lax-Milgran Theorem, there exists a unique $(u_2,v_2) \in \h$ satisfying

$$B((u_2,v_2), (w,z)) = \langle \mathcal{L}, (w,z)\rangle, \,\,\forall (w,z) \in \h.$$

So, $(u_2,v_2)$ verifies (\ref{u2}) and (\ref{v2}).  Defining $u_1$ and $v_1$ as in (\ref{3.6}) and $\eta$ as in (\ref{3.5}), we conclude that $((u_1,v_1),(u_2,v_2),\eta)\in \D(\mathcal{A})$ and it satisfies (\ref{3.2}), which proves that $R(I-\mathcal{B})=X$.  Then, $-\mathcal{B}$ is a maximal monotone operator and $D(\mathcal{B})$ is dense in $\h$.  Consequently, $\mathcal{B}$ is m-dissipative.

Recalling the theory  of linear semigroups (see e.g. Pazy \cite{Pazy}),   we conclude
the proof of Theorem \ref{theo-existence}.

\end{proof}

\section{The Exponential Stability}
In order to prove the observability inequality, we need the additional hypothesis
\begin{ass}\label{ass:2}
There exists $c>0$ such that $g(s) \leq -c g'(s)$.
\end{ass}
%\begin{enumerate}[label=(\subscript{H}{{\arabic*}})]
%\item $g\in C^1(\R^+) \cap L^1(\R^+)$ is nonnegative;
%\item $k_0=\|g\|_{L^1(\R^+)}< 1$;
%\item There exists $c>0$ such that $g(s) \leq -c g'(s)$;
%\item $b\in C^\infty(\Omega_1)$ is nonnegative, $\|b\|_{L^\infty(\overline{\Omega}_1)} \leq 1$ and there exists a constant $C>0$ such that \[
%\min_{x\in \overline{\Omega}_1} \{1-k_0 b(x)\} \geq C > 0. 
%\]
%\end{enumerate}
Assumption \ref{ass:1} ensures that the propagation speed is positive in \eqref{AUP}. Assumption \ref{ass:2} is classical in the study of the exponential decay of the energy with a memory term. 

We shall prove the following, which relies on assumptions and construction to be presented below.
\begin{theorem}\label{main}
Assume $k_2>k_1$ and that Assumption 2.1, 3.1, \ref{x0} and \ref{weakgcc} hold. Moreover, assume that $\Omega$ and $\Omega_2$ are strictly convex and that $\Omega_1^f$ satisfies the uniformly escaping geometry condition. Then there exists $\lambda>0$ and $C>0$ such that for all initial data 
\[
E(t)\leq C e^{-\lambda t} E(0).
\]
\end{theorem}

The proof of Theorem \ref{main} is done in the spirit of \cite{OA}. We begin by proving the weak observability inequality by contradiction. We obtain a sequence of solutions contradicting the weak observability to which we attach a microlocal defect measure. In the framework of the classical wave equation, the classical propagation results on the defect measure shows that the weak observability inequality holds if the support of the damping term satisfies the geometrical control condition (GCC). Compacity and unique continuation arguments are then used to prove that the weak observability inequality implies the observability inequality. In turn, the observability inequality is sufficient to deduce the exponential decay of the energy.

In the case of a transmission problem, the situation is more delicate. Under the hypothesis of Theorem \ref{main}, we shall prove that the rays of the optic geometry encountering the interface between the two medium satisfies the Snell's law
\begin{equation}\label{Snellintro}
\dfrac{\sin \theta_1}{\sqrt{k_1(1-k_0b(x))}}=\dfrac{\sin \theta_2}{\sqrt{k_2}},
\end{equation}
which simplifies to 
\begin{equation}\label{Snellintro}
\dfrac{\sin \theta_1}{\sqrt{k_1}}=\dfrac{\sin \theta_2}{\sqrt{k_2}},
\end{equation}
if the ray encounter the interface outside of $\textrm{supp}(b)$. A ray encountering the interface, say from $\Omega_1$ with an angle of incidence $\theta_1$, is then reflected at an angle $\theta_1$ and transmitted at an angle $\theta_2$ if \eqref{Snellintro} is not vacuous. By linearity of \eqref{AUP}, an interfering ray incoming from $\Omega_2$ may also exist. Indeed, if the ray possesses the same angle of incidence $\theta_2$, than the ray is reflected in $\Omega_2$ at an angle $\theta_2$ and transmitted at an angle $\theta_1$. Therefore, interference may occur between the two incoming rays such that the energy, or the support of the defect measure, concentrates along one of the outgoing rays. Therefore, observing only one of the two outgoing rays from the interface is not sufficient to gain information on the two incoming rays. However, we shall prove that if one of the two outgoing rays has no energy, then the energy has to have concentrated along the other outgoing ray (Proposition \ref{propag}), and that if the two outgoing rays possesses no energy, then the two incoming rays possess no energy as well (Corollary \ref{equivsuppmes}). 

By iterating the use of Proposition \ref{propag}, one can follow the propagation of the rays outgoing from the interface as long as one of the two outgoing rays is shown to be observed. We shall see that the concentration procedure can't last for so long if $\Omega_1 \setminus \textrm{supp}(b)$ satisfies the uniform escaping geometry condition introduced in \cite{LG}. Roughly speaking, this condition ensures that the rays propagating $\Omega_1 \setminus \textrm{supp}(b)$ uniformly escape from $\Omega_1 \setminus \textrm{supp}(b)$ to $\textrm{supp}(b)$ as well as the rays transmitted to $\Omega_2$. The uniform escaping geometry condition was stated in the context of boundary controllability and we shall extend the notion to the distributed control in Section 2.  

We conclude this part of the introduction to highlight a technique of proof used in this paper. We used Dafermos' change of variables to obtain an autonomous system \eqref{AUP} for which we could establish an observability inequality that allows us to conclude on the exponential decay. However, in many part of the proof, this change of variables will be deconstructed to work on the original, and more simple, system \eqref{MP1}. In some sense, Dafermos' change of variables gives insight into which observability inequality to prove. 

\subsection{The generalized bicharacteristics and the propagation theorem}

%
%We begin by recalling the propagation of the generalized bicharacteristics of \eqref{AUP}. We refer to \cite{Peppino} for a good presentation of the propagation for the classical wave equation and to \cite{LameBJ, LG} in the framework of a transmission problem. 

We begin by proving the weak observability, that is, there exists a constant $c(T)>0$ such that 
\begin{equation}\label{weakobs}
E(u,v,\eta)(0) \leq c(T)\ \left(\int_0^T\int_{0}^{\infty}(-g'(s)) \int_{\Omega_1}  b(x)|\nabla \eta(x,t,s)|^2 \, dxdsdt + \|(u_0,u_1,v_0,v_1, \eta_0)\|_{X^{-1}}^2 \right)
\end{equation}
where $X^{-1}$ is the usual dual space of $X$ for the wave equation \eqref{AUP}
\[
X^{-1}:= \L\times (\h )' \times L_g^2( \mathbb{R}^+;V'),
\]
with respect to the $L^2$ pivot space. 
 
We proceed by contradiction. Suppose that the weak observability does not hold. Therefore, there exists a subsequence of initial data $(u_0^n,u_1^n,v_0^n,v_1^n,\eta_0^n)\in X $ such that
\begin{align*}
\|(u_0^n,u_1^n,v_0^n,v_1^n, \eta_0^n)\|_{X}=1, \qquad \forall n\in \N, 
\end{align*}
and
\begin{align}
& \|(u_0^n,u_1^n,v_0^n,v_1^n, \eta_0^n)\|_{X^{-1}} \longrightarrow 0,  \nonumber  \\
\int_0^T\int_{0}^{\infty} (-&g'(s)) \int_{\Omega_1}  b(x)|\nabla \eta^n(x,t,s)|^2 \, dxdsdt \longrightarrow 0, \label{strgconveta}
\end{align}
since, from the hypothesis on $g$ and $b$, all the quantities in the right-hand side of \eqref{weakobs} are positive. In particular, we conclude that the sequence $(u^n,v^n)$, associated to the sequence of initial data, weakly converge to zero in $H^1((0,T)\times \Omega)$. Therefore, up to the extraction of a subsequence (using the same notations for the extracted sequence), there exist defect measures on $S^* \hat{\Sigma}_i$ (we postpone the definition of the cosphere bundle in the next subsection),
\[
(R_1 u^n,u^n) \longrightarrow \left< \mu,\kappa(R_1) \right>,  \quad (R_2v^n,v^n) \longrightarrow \left< \nu , \kappa(R_2) \right>,   
\]
for $R_i \in \psi^0(\R \times \overline{\Omega}_i)$, where $\kappa(R_i)$ is understood as a continuous function on $S^* \hat{\Sigma}_i$ (\cite{Peppino}).

Let us recall classical results on the propagation of the defect measure. 

\begin{theorem}\label{meschar}
Let $P$ be the classical wave operator over $\Omega$ and let $(u_n)$ be a bounded sequence of $L^2_{loc}(\Omega)$ weakly converging to zero and admitting a microlocal defect measure $\mu$. The following are equivalent 
\begin{itemize}
\item $Pu_n \longrightarrow 0$ strongly in $H^{-2}_{loc}(\Omega)$;
\item $supp(\mu) \subset Char(p)$. 
\end{itemize}
where $\sigma_2(P)=p$
\end{theorem}

\begin{theorem}\label{mesinv}
Let $P$ be the classical wave operator over $\Omega$, satisfying $P=P^*$, and let $(u_n)$ be a bounded sequence of $L^2_{loc}(\Omega)$ weakly converging to zero and admitting a microlocal defect measure $\mu$. Assume $Pu_n \longrightarrow 0$ in $H^{-1}_{loc}(\Omega)$. Then, for every $a \in C^\infty(\Omega \times (\R^d \setminus {0}))$ homogeneous of degree $-1$ in the second variable and of compact support in the first variable. Then  
\[
\int_{\Omega\times \R \times S^{d}} \{p,a\} d\mu =0.
\]
\end{theorem}

We shall prove that the framework we consider here fall in the scope of Theorem \ref{meschar} and Theorem \ref{mesinv}. 

In order to do so, we begin by proving that the contradiction argument implies that the memory term of \eqref{AUP} goes to zero strongly to zero in $H^{-1}_{loc}$. We begin by proving the following

\begin{lemma}\label{convetaen}
The strong convergence given by \eqref{strgconveta} implies 
\[
\int_0^\infty g(s) \int_{\Omega_1} b(x)|\nabla \eta^n(x,t,s)|^2\,dxds \longrightarrow 0, \qquad \textrm{for every } t\in [0,T].
\] 
\end{lemma}

\beginpf

This comes from the hypothesis $g(s) \leq -c g'(s) $, the strong convergence \eqref{strgconveta} and from the well-posedness result which implies that  
\[
\int_0^\infty g(s) \int_{\Omega_1} b(x)|\nabla \eta^n(x,t,s)|^2\,dxds
\]
is in $C([0,T])$. 

\endpf

From Lemma \ref{convetaen}, we obtain the strong convergence of $u^n$ over $\textrm{supp}(b)$.

\begin{lemma}\label{strongcvu}

From Lemma \ref{convetaen}, we have 
\[
\int_{supp(b)} | u_t^n |^2 + b(x) | \nabla u^n|^2  dx \longrightarrow 0 
\]
\end{lemma}

\beginpf

The proof comes directly from \cite{AMO} by replacing $\Delta$ by $div( b(x) \nabla)$ (see also \cite{Cavalcanti2} for a very similar proof). Since the proof is verbatim the same, it will be omitted.

\endpf

From Lemma \ref{strongcvu}, we conclude that Theorem \ref{meschar} and Theorem \ref{mesinv} applies for the first equation of \eqref{MP1}. Indeed, define 
 \begin{equation}\label{EqP1}
P_1 u := u_{tt} - k_1 \Delta u = -k_1\int_{-\infty}^t g(t-s) \textrm{ div}(b(x) \nabla u)(s)\,ds. 
\end{equation}
Let $\psi \in C_c^\infty((0,T)\times \Omega_1)$. Then , we have 
\begin{align*}
-\int_{\Omega_1} k_1\int_{-\infty}^t g(t-s) \textrm{ div}(b(x) \nabla u^n)(s)\,ds \psi u^n dx =&  \int_{\Omega_1} k_1\int_{-\infty}^t g(t-s) b(x) \nabla u^n(s)\,ds \nabla (\psi u^n) dx   \longrightarrow 0
\end{align*}
which gives $P_1 u^n \rightarrow 0$ in $H^{-1}_{loc}((0,T)\times \Omega_1)$ (outside the support of $b$, the right-hand side of \eqref{EqP1} is identically zero). Moreover, one readily obtain from Lemma \ref{convetaen} that the sequence of initial data $\eta^n_0$ strongly converge to $0$ in $L^2(\R^+;V)$. It remains to treat the sequence of initial data $(u_0^n,u_1^n,v_0^n,v_1^n)$, and, in particular, the propagation of the defect measure across the interface. 

Notice that, outside of the support of $b$, $P_1 u$ writes 
\[
P_1 u =  u_{tt} - k_1 \Delta u, 
\]
and therefore, away from the boundary and from the support of $b$, the support of the defect measure $\mu$ propagates along a union of bicharacteristics (given by the wave operator $\partial_t^2 - k_1 \Delta$). Moreover, we proved that the strong convergence of the viscoelastic term implies the strong convergence of $u^n$ over the support of $\textrm{supp}(b)$ in $H^1((0,T)\times \textrm{supp}(b))$, which, in turn, implies that the support of the defect measure $\mu$ is located outside the support of $b$.

It remains to describe how the bicharacteristics of \eqref{AUP} propagates in $\Omega$ and how the support of the defect measure propagates at the interface.  

\subsection{Propagation of the generalized bicharacteristics}

Let $M_i=\R_t \times \Omega_i, i=1,2$. We consider $T^*(M_i)=\{ (t,s,x,\tau,\sigma,\xi_i) \in \R^8 \, | \, (t,x)\in \R_t \times \Omega_i \}$. The principal symbol of $P_2=\d_t^2-k_2 \Delta$ is given by 
\[
p_2(t,x;\tau,\xi_2)=\tau^2 - k_2 |\xi_2|^2.
\]
The rays of the bicharacteristics in $\Omega_2$ are solution to 
\begin{equation}\label{charpro}
\begin{cases}
\dot{t}(s)=-2 \tau,  \quad \quad \dot{\tau}(s)=0, &  s \in \R, \\
\dot{x}(s)=2 k_2 \xi_2,  \, \quad \dot{\xi}_2(s)=0, &   s \in \R,
\end{cases}
\end{equation}
and we define the bicharacteristics rays as the projection over on the $(x,t)$ coordinates. It comes from \eqref{charpro} that the rays propagates in straight line and at constant speed away from the boundary. The characteristics set is defined as $\textrm{Char}(p_2)=\{ \rho_2=(x,t,\xi_2,\tau)\in T^*(M_2) \, | \, p_2(\rho_2)=0 \}$.  

On the other hand, according to Theorem \ref{meschar}, Theorem \ref{mesinv} and Lemma \ref{strongcvu}, the propagation is given by 
\[
P_1 u=u_{tt}-k_1 div( (1-k_0 b(x)) \nabla u).\]
The principal symbol of $P_1$ is
\[
\sigma_2(P_1)=p_1=\tau^2-k_1 (1-k_0b(x))|\xi_1|^2.
\]
Hence, the propagation of the bicharacteristics in $\Omega_1$ is given by
\[
\begin{cases}
\dot{t}(s)=-2 \tau, \qquad \qquad \qquad \, \, \quad \quad \dot{\tau}(s)=0, & s \in \R, \\
\dot{x}(s)=2k_1 (1-k_0b(x))\xi_1(s), \quad  \dot{\xi}_1(s)=k_1 k_0 \nabla b(x)\xi_1(s), &  s \in \R.
\end{cases}
\]
It is important to notice that, outside the support of $b$, the rays of the optic geometry propagate in straight line as in $\Omega_2$, that is 
\[
\begin{cases}
\dot{t}(s)=-2 \tau,  \quad \quad \dot{\tau}(s)=0, & s \in \R, \\
\dot{x}(s)=2k_1 \xi_1,  \, \quad \dot{\xi}_1(s)=0, &  s \in \R.
\end{cases}
\] 
Let us finally define the characteristic set $\textrm{Char}(p_1)=\{ \rho_1=(x,t,\xi_1,\tau)\in T^*(M_1) \, | \, p_1(\rho_1)=0 \}$.  \newline

\textbf{Propagation near $\d \Omega$}\newline

In this section, we follow closely the presentation of \cite{Peppino}. At the boundary $\d \Omega$, a stardard reflection occurs. Let the local geodesic coordinates $x=(x',x_n)$ such that $\Omega=\{x_n >0\}, \d \Omega=\{ x =0 \}$ and where $x'$ is the tangential component near the origin. The Laplacian takes locally the form
\[
\Delta=\d_{x_n}^2 + R(x_n,x',D_{x'}), 
\]
where $R(x_n,x',D_{x'}) $ is a second order tangential elliptic operator of real principal symbol $r(x_n,x',\xi'_1)$. Let $T^*(\d \Omega \times \R_t )=\{ \rho \in T^*(M_1) \, | \, x_n=0 \}$. We recall the definition of the compressed cotangent bundle : for $x$ near the boundary, we define ${}^bT(\overline{M}_1)=\cup_{x\in \overline{\Omega}_1} {}^b T_x (\overline{M}_1)$. We have the map 
\begin{align*}
j:  T^*M_1 &  \longrightarrow {}^b T^* M_1 = \cup_{x\in \overline{\Omega}_1} \left( {}^b T_x \overline{M}_1 \right)^* \\ 
 (x_n,x',t,\xi_1^n,\xi_1',\tau) & \longmapsto  (x_n,x',t,x_n \xi_1^n,\xi_1',\tau).
\end{align*}
Near the boundary, we have $\Sigma_1 := j(\textrm{Char}(p_1)) \subset {}^bT^* M_1 $ and $\Sigma_1^0 := \restriction{\Sigma}{x_n=0} \subset {}^b T^*\restriction{M_1}{x_n=0} \simeq T^*(\d \Omega \times \R_t)$. We define the glancing set
\[
\G:=\left\{(t,x,\tau,\xi_1)\in \Sigma_1^0   \, | \, r(x_n,x',\xi_1)=|\tau|/\sqrt{k_1} \right\},
\]
and the hyperbolic set $H=\Sigma_1^0 \setminus G$ defined by 
\[
\H:=\left\{(t,x,\tau,\xi_1)\in \Sigma_1^0  \, | \, r(x_n,x',\xi_1)>|\tau|/\sqrt{k_1} \right\}.
\]
Finally, define the elliptic set 
\[
\E:=\left\{(t,x,\tau,\xi_1)\in \Sigma_1^0  \, | \, r(x_n,x',\xi_1)<|\tau|/\sqrt{k_1} \right\},
\]
define $\hat{\Sigma}_1=\Sigma_1 \cup \E$ and $S^* \hat{\Sigma}_1=\hat{\Sigma}_1 / (0,\infty)$. It is well known that a ray $\gamma^-(s)$ encountering the boundary $\d \Omega$ transversally at a point $\rho^- \in \H$ is reflected with the same angle as the angle of incidence and that the reflected ray $\gamma^+(s)$ corresponds to $\rho^+ \in \H$ such that $j(\rho^-)=j(\rho^+)$ since $x_n=0$. The fact that $\rho^+$ exists comes from the definition of the set $\H$ which ensures that $p_1(\rho)=0$ has two real roots. We also know that the propagation of the bicharacteristics associated to the set $\G$ depends on the nature of the point $\G$. The rays may glide, hit non-transversally the boundary or encounter tangentially the boundary and glide on the boundary (see \cite{LameBJ}). We recall here the crucial hypothesis that there is no contact of infinite order between the geodesics and the boundary so that the bicharacteristic flow is uniquely defined. \newline

\textbf{Propagation near $\d \Omega_2$} \newline

Near $\d \Omega_2$, the propagation of the generalized bicharacteristics was described in \cite{LG} (see also \cite{LameBG,LameBJ}). We use the local geodesic coordinates near $\d \Omega_2$ such that, locally, $\Omega_2=\{ x_n >0 \}$, $\d \Omega_2 = \{ x_n = 0 \}$ and $\Omega_1=\{ x_n<0 \}$. One deduce the Snell's law from the trace equality at the interface $u=v$. Indeed, taking the tangential gradient $\nabla'$ yields $\nabla' u = \nabla' v$, which translates to $\xi_1'=\xi_2'$ or, using the characteristic set, to
\begin{equation}\label{snell}
\dfrac{\sin \theta_1}{\sqrt{k_1(1-k_0b(x))}}=\dfrac{\sin \theta_2}{\sqrt{k_2}}.
\end{equation}
This condition is simplifies to 
\begin{equation*}
\dfrac{\sin \theta_1}{\sqrt{k_1}}=\dfrac{\sin \theta_2}{\sqrt{k_2}},
\end{equation*}
where $b(x)=0$ for $x\in \d \Omega_2$. According to the hypothesis $k_1<k_2$ and $k_1(1-k_0b(x))<k_2, \forall x\in \overline{\Omega}_2$, we see that \eqref{snell} is never vacuous for any $\theta_2 \in [0,\pi/2]$ and that there exists a critical angle $\theta_1^c$, that depends on $x$ for $b(x)\neq 0$ for $x\in \d \Omega_2$, such that \eqref{snell} is satisfied for $\theta_2=\pi /2$. This angle geometrically corresponds to the transmission of a gliding ray (recall that we assume $\Omega_2$ strictly convex). This also corresponds to the decomposition of the phase space at the interface $T^*((0,T) \times \d \Omega_2)$. Indeed, this set is decomposed as 
\begin{align*}
\H^1 \times \H^2:=&\left\{ (\rho_1,\rho_2) \in T^*((0,T) \times \d \Omega_2)\, | \, r(x_n,x',\xi_i) > |\tau|/\sqrt{k_i} \right\}, \\
\H^1 \times \G^2:=&\left\{ (\rho_1,\rho_2) \in T^*((0,T) \times \d \Omega_2)\, | \, r(x_n,x',\xi_1) > |\tau|/\sqrt{k_1},   r(x_n,x',\xi_2) = |\tau|/\sqrt{k_2} \right\}, \\
\H^1 \times \E^2:=&\left\{ (\rho_1,\rho_2) \in T^*((0,T) \times \d \Omega_2)\, | \, r(x_n,x',\xi_1) > |\tau|/\sqrt{k_1},   r(x_n,x',\xi_2) < |\tau|/\sqrt{k_2} \right\}, \\
\G^1 \times \E^2:=&\left\{ (\rho_1,\rho_2) \in T^*((0,T) \times \d \Omega_2)\, | \, r(x_n,x',\xi_1) = |\tau|/\sqrt{k_1},   r(x_n,x',\xi_2) < |\tau|/\sqrt{k_2} \right\}, \\
\E^1 \times \E^2:=&\left\{ (\rho_1,\rho_2) \in T^*((0,T) \times \d \Omega_2)\, | \, r(x_n,x',\xi_i) < |\tau|/\sqrt{k_i} \right\},
\end{align*}
\begin{figure}[!ht]
\begin{center}
	\includegraphics[height=3.5cm]{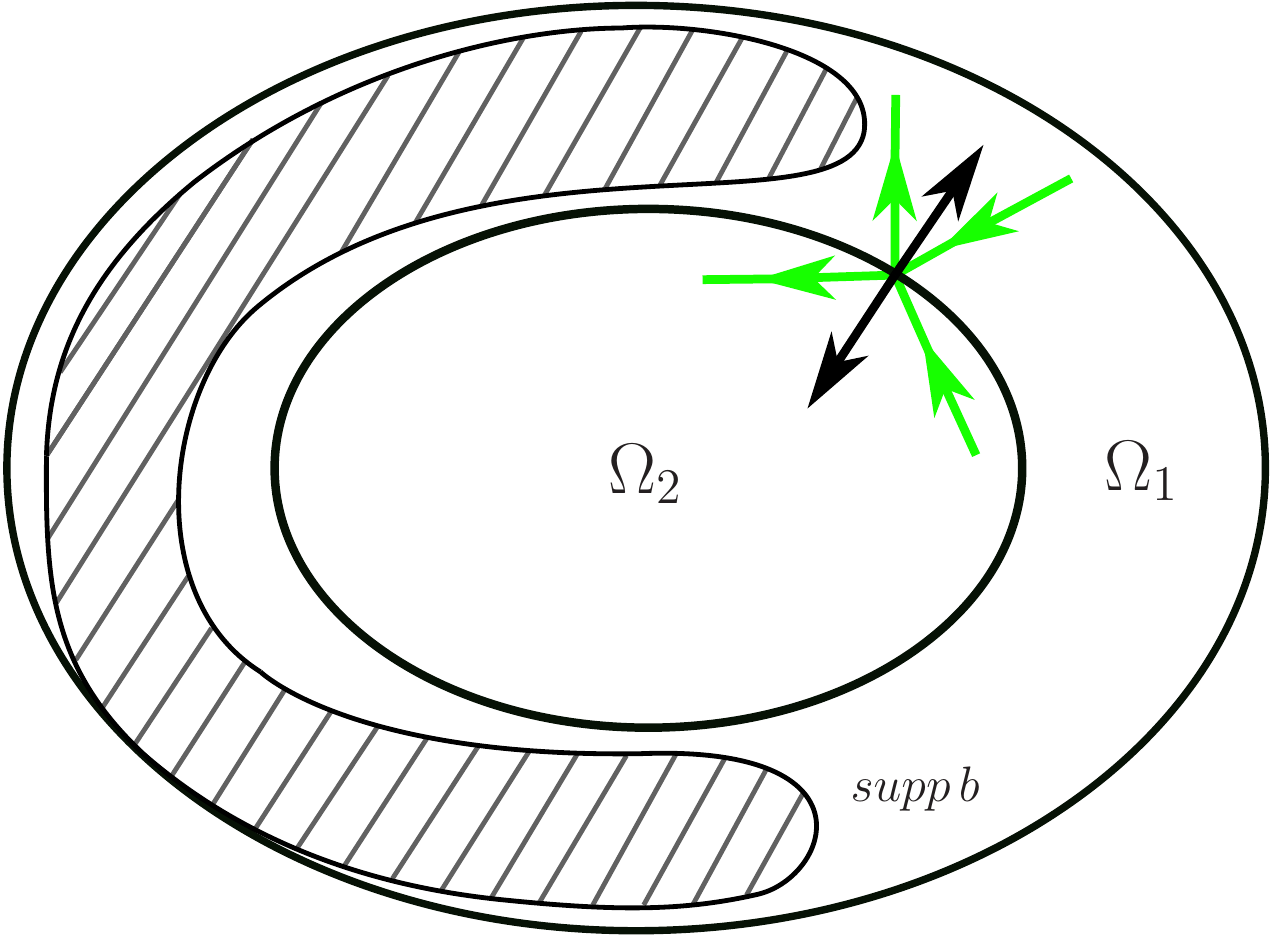}
	\includegraphics[height=3.5cm]{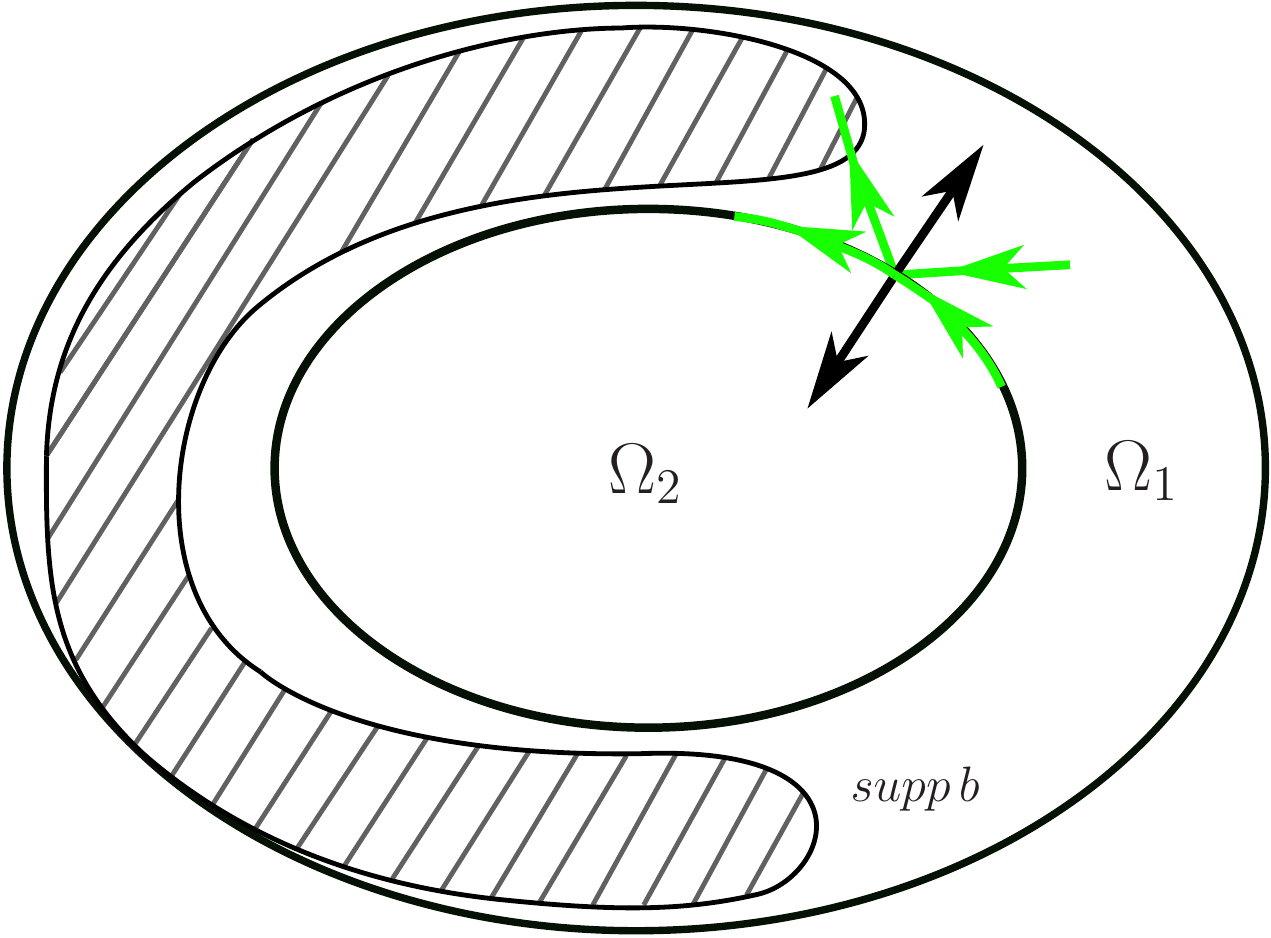} \\
	\includegraphics[height=3.5cm]{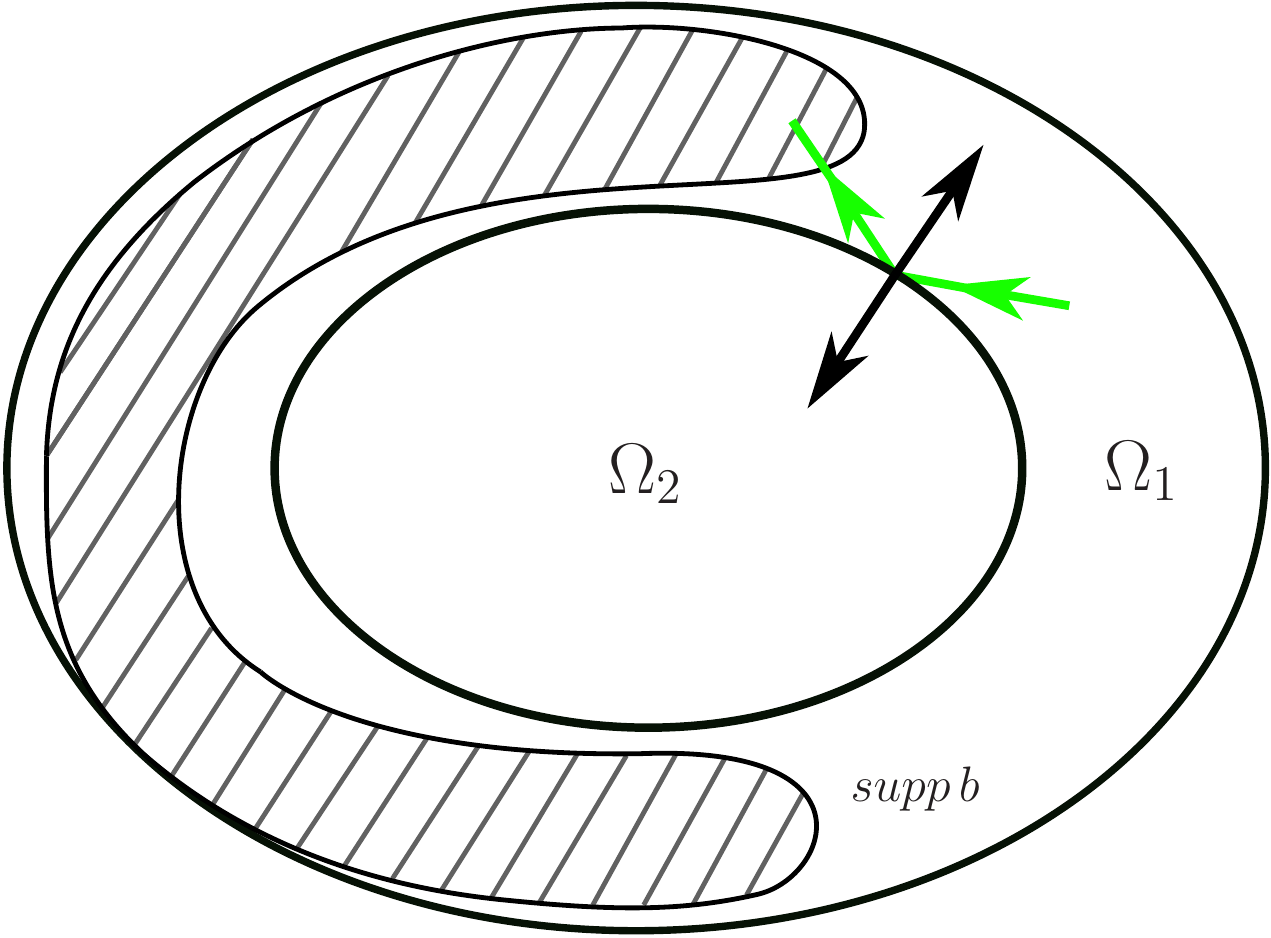}
	\includegraphics[height=3.5cm]{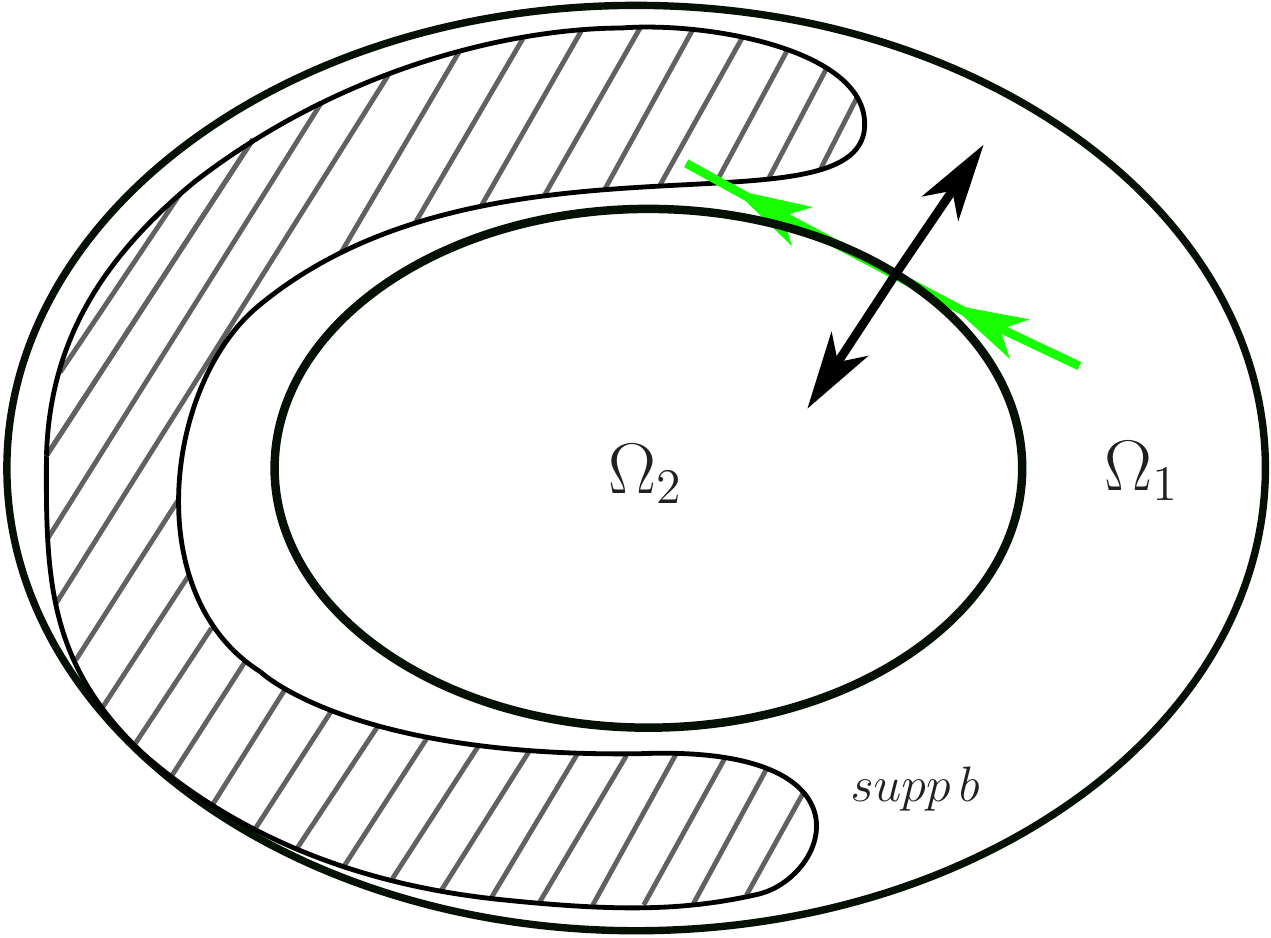}
\end{center}
\caption{\footnotesize{Representation of rays at the interface for $\H^1 \times \H^2$ (up left), $\H^1 \times \G^2$ (up right), $\H^1 \times \E^2$ (down left) and $\G^1 \times \E^2$ (down right) } }\label{Gam1}
\end{figure}
The propagation near the interface is similar to the one near $\d \Omega$. Suppose for instance that a ray $\gamma_1^-(s)$ encounters the interface at a point $\rho_1^-$ and at an angle $\theta_1$. From the classical properties, if $\theta_1 < \pi/2$, then $\rho_1^- \in \H^1$ and there exists $\rho_1^+ \in \H^1$ such that $j(\rho_1^-)=j(\rho_1^+)$. Therefore $\gamma_1^-(s)$ is reflected in $\Omega_1$ in a ray $\gamma_1^+(s)$. Moreover, if the angle $\theta_1<\theta_1^c$, then there exists $(\rho_2^-,\rho_2^+)\in (\H^2)^2$. The rays associated to these points correspond to the transmission and the possible interference of the ray. If $\theta_1=\theta_1^c$, then there is a reflection and a ray transmitted tangentially. Finally, for points $(\rho_1,\rho_2)\in (\H^1 \cup \G^1) \times \E^2$, the ray in $\Omega_1$ stays in $\Omega_1$ and is not transmitted. We highlight here, as it is crucial in the geometrical argument, that every ray of $\Omega_2$ intersecting (transversally) the interface are transmitted to $\Omega_1$.

\subsection{Properties of the defect measures}

We recollect what we have proved so far for the defect measure $\mu$ and $\nu$ : their support is included in $\textrm{Char}(p_i)$ and is invariant along the bicharacteristic flow inside $\Omega_1$ and $\Omega_2$ (Theorem \ref{meschar} and Theorem \ref{mesinv}). Moreover, Lemma \ref{strongcvu} implies $\mu = 0 $ over $T^*(\textrm{supp}(b) \times (0,T))$. The reflection on the outside boundary $\d \Omega$ is understood since \cite{BLR}.
\begin{lemma}
For $\rho \in \H \cup \G$, we have 
\[
\left( \gamma^- \cap \textrm{supp}(\mu)\right) = \emptyset \Leftrightarrow \left( \gamma^+ \cap \textrm{supp}(\mu)\right) = \emptyset 
\]
\end{lemma}
It remains to understand the propagation of the defect measure across the interface. If $\textrm{supp}(b) \, \cap \, \d \Omega = \emptyset$, then the following readily apply (\cite{LG} by adapting the proof of \cite{LameBG}). 

\begin{prop}\label{propag}
With the above notations, we have 
\begin{enumerate}
\item If $(\rho_1,\rho_2) \in (\H^1 \cup \G^{1,-} \cup \E^1) \times \E^2$, then $\nu=0$ near $\rho_2$. Therefore,
\begin{enumerate}
\item if $\rho_1 \in \E^1$, then $\mu=0$ near $\rho_1$,
\item otherwise, $\rho_1 \in \H^1 \cup \G^{1,-}$ and the support of $\mu$ propagates from $\gamma_1^-$ to $\gamma_1^+$. 
\end{enumerate}
\item otherwise $\rho_2\in \H^2 \cup \G^{2,+}$ and $\rho_1 \in \H^1$. In this case, if $\gamma_1^- \cap \textrm{supp}(\mu) = \emptyset$ (resp. $\gamma_2^- \cap \textrm{supp}(\nu) = \emptyset$), then the support of $\nu$ (resp. $\mu$) propagates from $\gamma_2^-$ (resp. $\gamma_1^-$) to $\gamma_i^+, i=1,2$.
\end{enumerate}
\end{prop}

\begin{corollary}\label{equivsuppmes}
For $\rho_2\in \H^2 \times \G^2$ such that the intersection of the bicharacteristic rays $\gamma_i^{\pm}$ is non-diffractive, we have the following equivalence
\begin{align*}
\left((\gamma_1^-)\cap \textrm{supp}(\mu_1)\right) &\cup \left((\gamma_2^-)\cap \textrm{supp}(\mu_2)\right) = \emptyset \\
&\Updownarrow \\
\left((\gamma_1^+)\cap \textrm{supp}(\mu_1) \right)&\cup\left( (\gamma_2^+)\cap \textrm{supp}(\mu_2)\right) = \emptyset. 
\end{align*}
\end{corollary}

If $\textrm{supp}(b) \cap \d \Omega \neq \emptyset$, then we use the following lemma to deal with the memory term in the boundary condition to adapt the proof of Proposition \ref{propag} and Corollary \ref{equivsuppmes}. Let us first define for $f\in L^2( \R \times \d \Omega_2)$
\begin{equation}\label{G}
Gf:= \int_{-\infty}^\infty g(t-s) b(x') f(x',t) ds,
\end{equation}
where $g$ is extended by zero to define $G$. We have
\begin{lemma}\label{invG}
Let $f\in L^2( \R \times \d \Omega_2)$, assume $b$ and $g$ satisfy the hypothesis of Theorem \ref{main} and $G$ is defined as in \eqref{G}. Then $I-G$ is invertible from $L^2( \R \times \d \Omega_2)$ to itself. 
\end{lemma}

\beginpf

The invertibility of $I-G$ comes from $\|Gf\|_{L^2( \R \times \d \Omega_2)}< \|f\|_{L^2( \R \times \d \Omega_2)}$. Indeed, from Young's inequality for the convolution, we have
\begin{align*}
\|Gf\|_{L^2( \R \times \d \Omega_2)} =& \left\| \int_{-\infty}^\infty g(t-s)b(x')f(x',t) ds \right\|_{L^2( \R \times \d \Omega_2)} \\ 
\leq & \|g\|_{L^1(\R)} \|b f\|_{L^2(\R \times \d \Omega_2)} \\
\leq & \|g\|_{L^1(\R)} \|b\|_{L^\infty(\d \Omega_2)} \| f\|_{L^2(\R \times \d \Omega_2)}.
\end{align*}
But from the hypothesis on $b$ and $g$, we have  $\|b\|_{L^\infty(\d \Omega_2)}\leq1$ and $\|g\|_{L^1(\R)} < 1$, hence the result. 

\endpf

We proceed to adapt the proof of Proposition \ref{propag}. 

\beginpf 

We consider the most technical case $(\rho_1,\rho_2)\in \H^1 \times (\H^2 \cup \G^{2,+})$, as the other cases follow using the ellipticity when $\rho_2\in \E^2$. We recall the procedure described in \cite[Appendix A.2]{LameBG}. Let the local geodesic coordinates $(x_n,x')$ are such that, locally, we have $\Omega_2=\{x_n >0\}, \d \Omega_2=\{x_n=0\}$ and $\Omega_1=\{x_n <0\}$. Recall that near this point, $u_k$ is strongly converging in $H^1((0,T)\times \Omega_1)$ and we have the expression of $P_1$ given by \eqref{EqP1}. Therefore, 
\[
P_i=k_i(\d_{x_n}^2+R(x_n,x',D_{x'}))
\]
where the principal symbol of $R$ is $r(x_n,x',\xi')$ and $r(x_n,x',\xi')\geq c|\xi'|^2$. One can then use \cite[Lemma A.1]{LameBG} to factorise the pseudodifferential operators in two different ways
\begin{align*}
P_i= k_i(D_{x_n}-\Lambda^+(x_n,x',D_{x'}))(D_{x_n}-\Lambda^-(x_n,x',D_{x'}))+T(x_n,x',D_{x'}) \\
P_i=k_i(D_{x_n}-\widetilde{\Lambda}^-(x_n,x',D_{x'}))(D_{x_n}-\widetilde{\Lambda}^+(x_n,x',D_{x'}))+\widetilde{T}(x_n,x',D_{x'}) 
\end{align*}
where $\Lambda^\pm$ and $\widetilde{\Lambda}^\pm$ are tangential pseudodifferential operators of order 1 and such that $\sigma_1(\Lambda^\pm)=\pm \sqrt{r}$ and $\sigma_1(\widetilde{\Lambda}^\pm)=\pm \sqrt{r}$ and where $T$ and $\widetilde{T}$ are tangential pseudodifferential operators of order $-\infty$. A microlocalisation near of $\gamma_i^{\pm}$ is done using $q_0(x',\xi_i')$ a symbol of order $0$ and equal to $1$ in a conical neighborhood of $\rho_i$ and of compact support. Let us remark that at a point $(\rho_1,\rho_2)\in \H^1 \times (\H^2 \cup \G^{2,+})$, the tangential components of $\rho_1=(x_n,x',\xi_1^n,\xi_1')$ and $\rho_2=(x,y,\xi_2^n,\xi_2')$ are equal : $(x',\xi_1')=(x',\xi_2')$. Therefore, the same symbol $q_0$ may be used for the microlocalisation near $\rho_1$ and $\rho_2$. The symbol is then propagated by the Hamiltonian 
\[
(c_i\d_{x_n} \mp H_{\sqrt{r}})q_i^{\pm}=0, \quad \restriction{q_i^{\pm}}{x_n=0}=q_0. 
\]
Consider $\varphi \in C_0^\infty( \R_x)$ to be equal to $1$ near $0$ and of compact support near $0$. If we denote $Q_i^\pm = \textrm{Op}(\varphi q_i^\pm)$ and $Q_0 = \textrm{Op}(\varphi q_0)$, then we obtain the same results as in \cite[Appendice A.2]{LameBG}, that is, if we consider bounded sequences $(u^k,v^k) \subset H^1(]-1,0[\times Y) \times H^1(]0,1[\times Y)$, where $Y=\{ x' \, | \, |x'| <1 \} $, such that 
\begin{align*}
P_1 u& \rightarrow 0 \textrm{ in } L^2(]-1,0[\times Y), \\
P_2 v& \rightarrow 0 \textrm{ in } L^2(]0,1[\times Y).
\end{align*}
and if we suppose, without loss of generality, that $\gamma_1^+ \cap \textrm{supp}(\mu) = \emptyset$, then (recall that we can deduce that $\restriction{u^k}{x=0}$ and $\restriction{D_x u^k}{x=0}$ are bounded in $H^1_{\rho_i}$ and $L^2_{\rho_i}$ respectively) 
\begin{equation}\label{reltrace}
k_1Q_0(\restriction{D_{x_n} u^k}{x_n=0}-\Lambda^- \restriction{u^k}{x_n=0}) \rightarrow 0 \textrm{ in } L^2(Y).
\end{equation}
We deduce the relation between the traces of $v$ by applying $\Lambda^-$ to the relation $u=v$ and by using Lemma \ref{invG} to obtain 
\[
k_1 \restriction{D_{x_n} u^k}{x_n=0} = k_2(I-G)^{-1}\restriction{D_{x_n} v^k}{x_n=0}. 
\]
These two relations together with \eqref{reltrace} implies 
\[
k_1Q_0\left( \frac{k_2}{k_1}(I-G)^{-1}\restriction{D_{x_n} v^k}{x_n=0}-\Lambda^- \restriction{v^k}{x_n=0}\right) \rightarrow 0 \textrm{ in } L^2(Y),
\]
which is a Lopatinski condition uniformly in $(\rho_1,\rho_2) \in \H^1 \times (\H^2 \cup \G^2)$. 

\endpf

We now turn ourselves to the proof of Corollary \ref{equivsuppmes}. 

\beginpf 

The proof of Corollary \ref{equivsuppmes} follow closely the proof of Proposition \ref{propag} with the assumption that two rays do not intersect the support of the microlocal defect measure. Assume for simplicity that $\gamma_1^\pm \cap \, \textrm{supp}(\mu) = \emptyset$. Then we obtain the relations
\begin{align*}
k_1Q_0(\restriction{D_{x_n} u^k}{x_n=0}-\Lambda^- \restriction{u^k}{x_n=0}) \rightarrow 0 \textrm{ in } L^2(Y), \\
k_1Q_0(\restriction{D_{x_n} u^k}{x_n=0}-\tilde{\Lambda}^+ \restriction{u^k}{x_n=0}) \rightarrow 0 \textrm{ in } L^2(Y),
\end{align*}
from which we deduce 
\begin{align}
Q_0(\restriction{D_{x_n} u^k}{x_n=0})  \rightarrow 0 \textrm{ in } L^2(Y), \label{neuconv}\\
Q_0((\Lambda^- - \tilde{\Lambda}^+) \restriction{u^k}{x_n=0}) \rightarrow 0 \textrm{ in } L^2(Y). \label{tangconv}
\end{align}
Together with Lemma \ref{invG}, we deduce from \eqref{neuconv}
\[
Q_0((I-G)^{-1}\restriction{D_{x_n} v^k}{x_n=0})  \rightarrow 0 \textrm{ in } L^2(Y).
\]
Notice that $[Q_0,(I-G)^{-1}]$ is a tangential pseudodifferential operator of order $-1$. Therefore, by commuting $Q_0$ with $(I-G)^{-1}$ and inverting once again $(I-G)^{-1}$ implies 
\[
Q_0(\restriction{D_{x_n} v^k}{x_n=0})  \rightarrow 0 \textrm{ in } L^2(Y),
\]
Using \eqref{tangconv}, we obtain 
\begin{align*}
k_1Q_0(\restriction{D_{x_n} v^k}{x_n=0}-\Lambda^- \restriction{v^k}{x_n=0}) \rightarrow 0 \textrm{ in } L^2(Y), \\
k_1Q_0(\restriction{D_{x_n} v^k}{x_n=0}-\tilde{\Lambda}^+ \restriction{v^k}{x_n=0}) \rightarrow 0 \textrm{ in } L^2(Y),
\end{align*}
the desired result. 

\endpf

\subsection{Uniformly escaping geometry}

We detail here the geometrical argument to conclude on the weak observability. The geometrical argument relies on the construction done in \cite{LG}, which we adapt in the case of a distributed damping. We easily transfer the geometrical construction in the boundary case \cite{LG} to the distributed case in the following way. We define $\Gamma_1 \subset \d \Omega$ the points of $\d \Omega$ such that for all $x\in \Gamma_1$, the ray in the inward normal direction intersects transversally $\textrm{supp}(b)$ (see figure \ref{Gam1}). Since the interior of $\textrm{supp}(b)$ is assumed non-empty, $\Gamma_1$ is non-empty and open in $\d \Omega$. 
\begin{figure}[!ht]
\begin{center}
	\includegraphics[height=3.5cm]{domain2.pdf}
	\includegraphics[height=3.5cm]{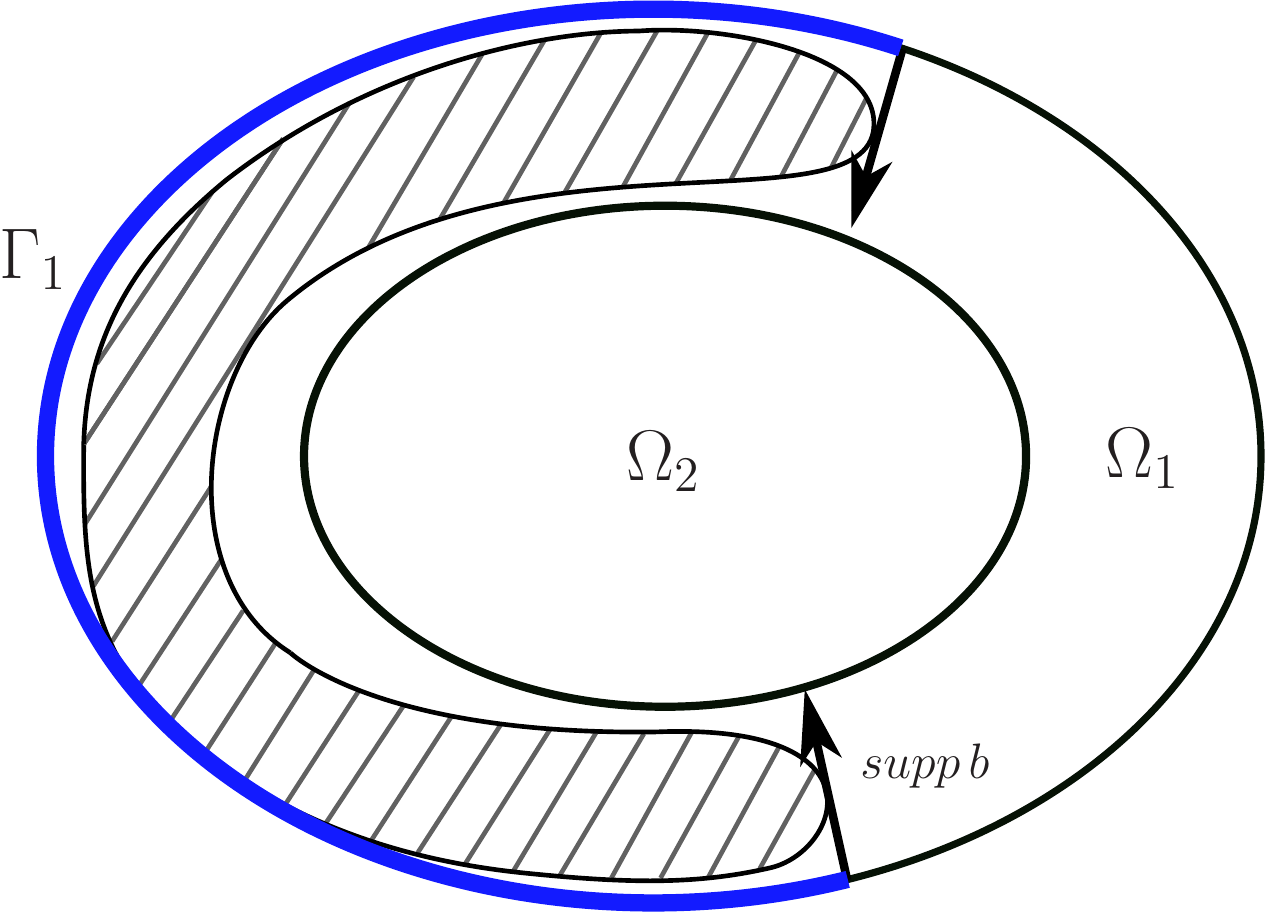}
\end{center}
\caption{\footnotesize{Construction of $\Gamma_1$}}\label{Gam1}
\end{figure}

We assume the following geometrical assumptions
\begin{ass}\label{x0}
There exists $x_0$ such that $\Gamma_1=\Gamma(x_0)$ where 
\begin{equation}\label{defgammax0}
\Gamma(x_0):=\{ x\in \d \Omega \, | \, \<(x-x_0),n(x)\> >0 \}.
\end{equation}
\end{ass}
\begin{ass}\label{weakgcc}
Consider $\rho^+ \in \H$ such that $\Pi_x(\rho^+) \in \Gamma_1$. Then, one of the two rays $\gamma^+$ or $\gamma^-$ associated to $\rho^+$ or $\rho^-$ intersects transversally $\textrm{supp}(b)$ before, eventually, intersecting transversally $\d \Omega_2$. 
\end{ass}
In the previous assumption, $\Pi_x$ denote the spatial projection of $\rho^+ \in T^*(\d \Omega \times (0,T))$ and $\rho^-$ is the coordinate of the reflection associated to $\rho^+ \in \H$, that is $j(\rho^+)=j(\rho^-)$.
 
Assumption \ref{x0} ensures that the uniformly escaping condition implies the weak observability. We highlight that Assumption \ref{x0} alone does not ensure the weak observability due to the interference phenomenon at the interface \cite{LG}. Assumption \ref{weakgcc} implies that the rays outgoing from $\Gamma_1$ do not intersect the support of $\mu$. Indeed, for every $\rho^+ \in T^*(\d \Omega \times (0,T))$, either the ray from $\rho^+$ or $\rho^-$ intersects $\textrm{supp}(b)$ after a finite number of reflection on $\d \Omega$. The results therefore follow from classical results on the wave equation with homogeneous boundary condition. The same holds for $\rho \in \G \times \E$. This allow us to consider $\Gamma_1$ as an observable boundary region, and then one can proceed with the construction \cite{LG} to verify if $\Gamma_1$ satisfies the uniform escaping condition. We recall that this construction consists to prove that $\Gamma_1$ implies the existence of $\Gamma_2 \subset \d \Omega_2$ such that every ray starting from this part of the interface do not contribute to the support of the mesure $\mu$ or $\nu$. Moreover, it is shown in \cite{LG} that Assumption \ref{x0} and \ref{weakgcc} and from the geometrical construction in \cite{LG}, that $\Gamma_2$ satisfies GCC for $\Omega_2$. 
\begin{figure}[!ht]
\begin{center}
	\includegraphics[height=3.5cm]{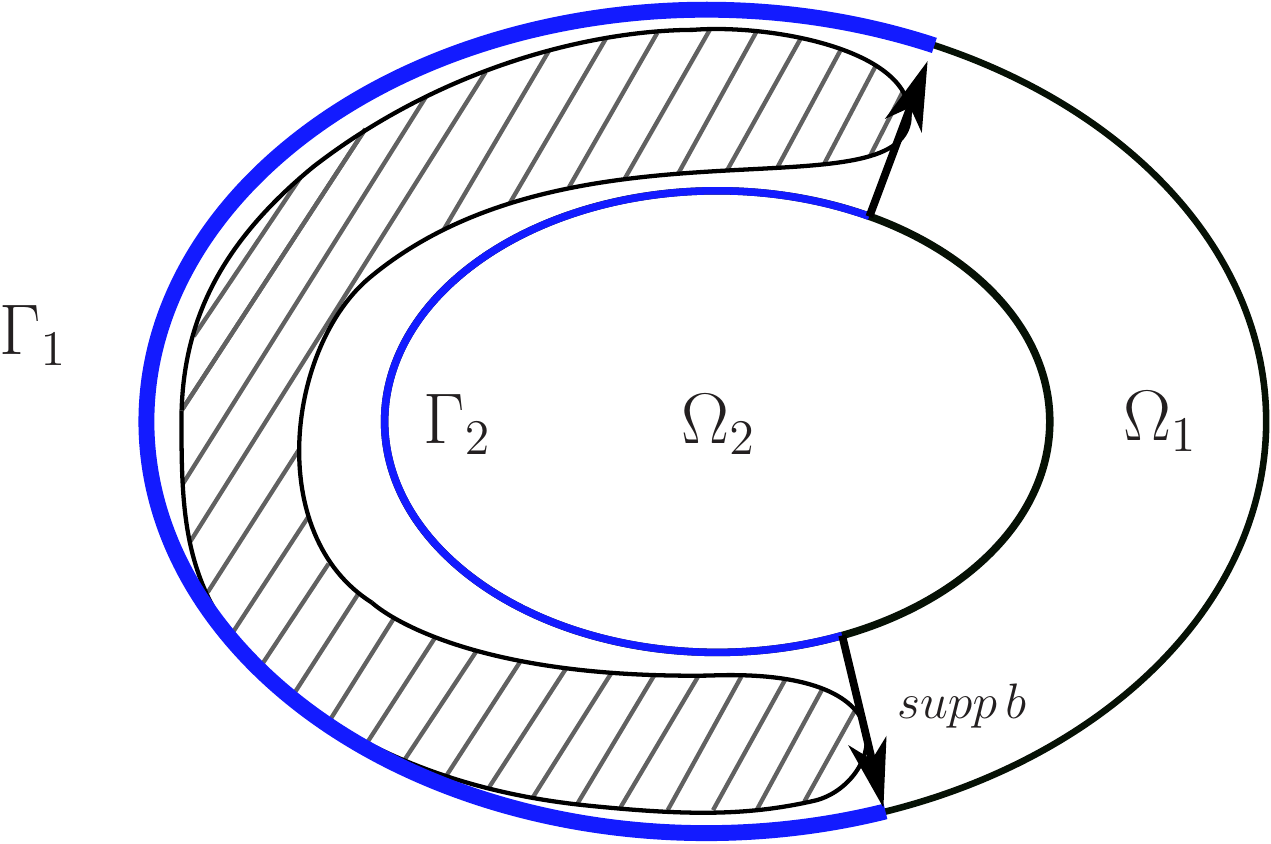}
\end{center}
\caption{\footnotesize{Representation of $\Gamma_2$ after the iterative construction in \cite{LG}}}
\end{figure}

\begin{figure}[!ht]
\begin{center}
	\includegraphics[height=3.5cm]{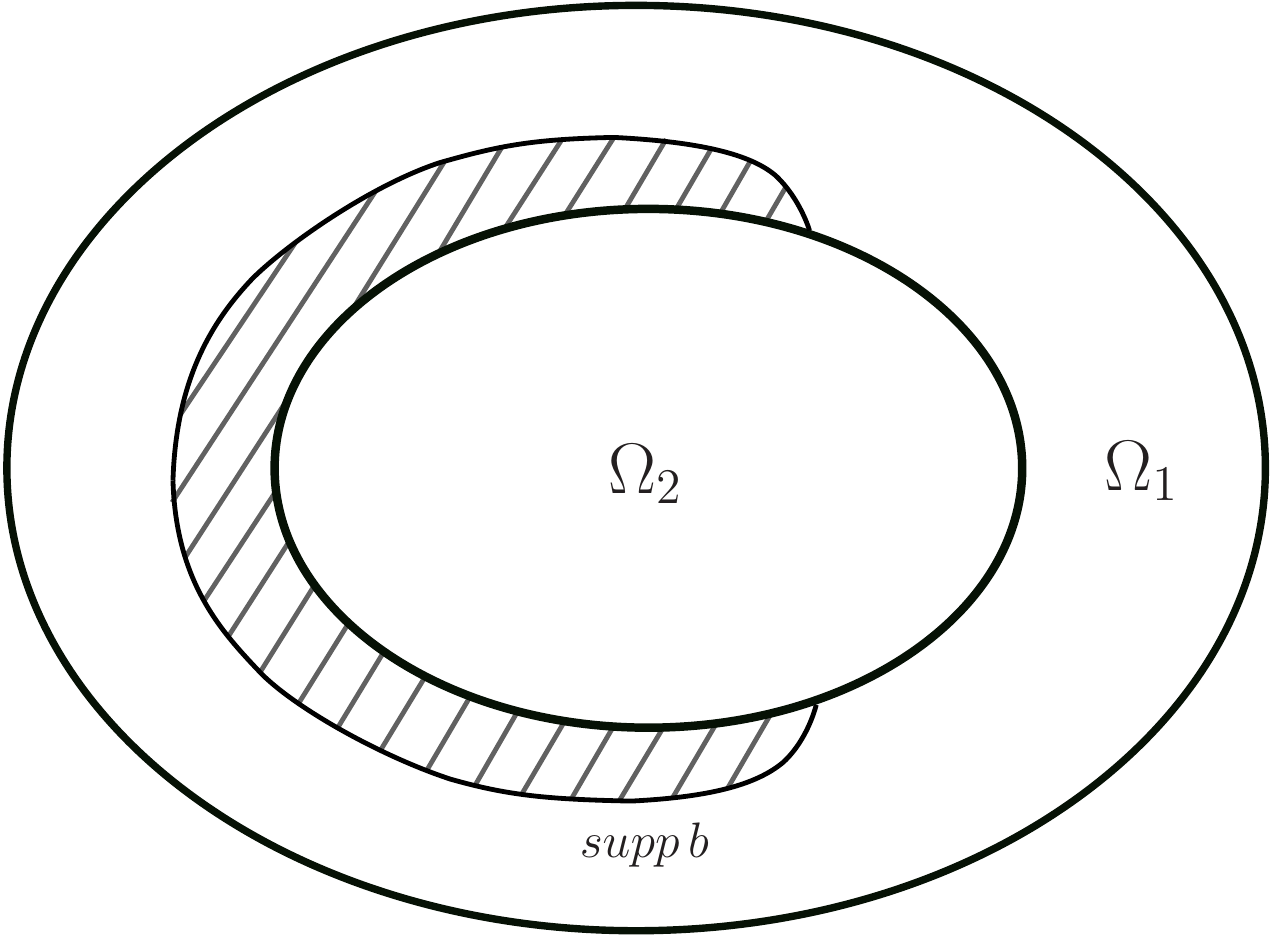} 
\end{center}
\caption{\footnotesize{Representation of a spatial domain for \eqref{MP1} in the case where the geometrical construction yields Assumption \ref{x0} but not Assumption \ref{weakgcc} and therefore for which we do not have observability due to the presence of rays propagating near $\d \Omega$ without encountering $supp(b)$.}}
\end{figure}

We then define, using the remaining part of the boundary $\d \Omega \setminus \Gamma_1$ and $\d \Omega_2 \setminus \Gamma_2$, a region $\Omega_1^f \subset \Omega_1$. We then say that $\Omega_1^f$ satisfies the uniformly escaping geometry condition if every ray from $\Omega_1^f$ escape uniformly to $\Omega_1 \setminus \Omega_1^f$ were the rays are observed. 
\begin{figure}[!ht]
\begin{center}
	\includegraphics[height=3.5cm]{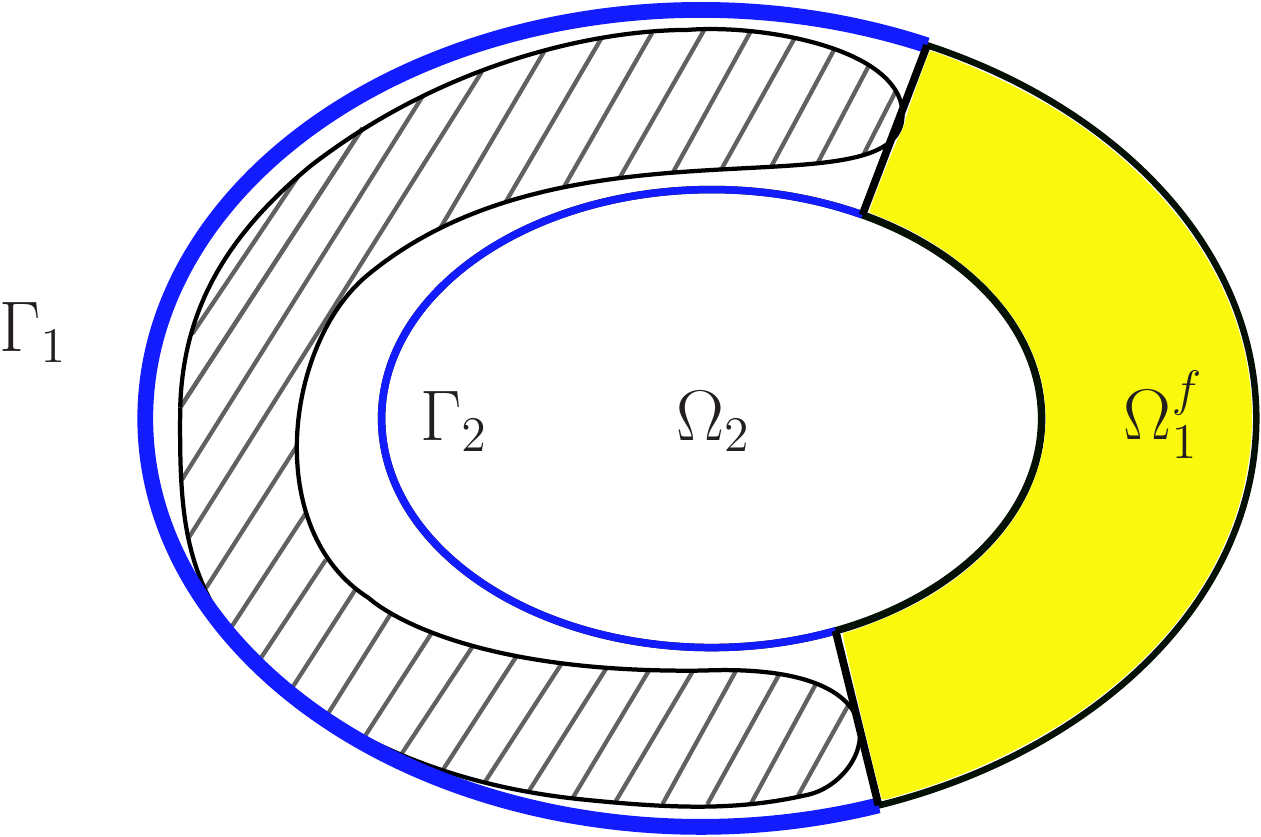}
\end{center}
\caption{\footnotesize{Representation of $\Omega_1^f$}}
\end{figure}
In order to recall precisely the uniformly escaping geometry condition, let us introduce the collision map in the billiard literature (\cite{Billiards}) for $\Omega_1$
\begin{align}\label{billardmap}
\F :  \left(\d \Omega \cup \d \Omega_2 \right)\times \R^2 & \longrightarrow \left(\d \Omega \cup \d \Omega_2 \right)\times \R^2, \\
 (x,\xi)\qquad  & \longmapsto \qquad (x^1,\xi^1), \nonumber
\end{align}
where $x^1$ is the point where the ray of $\Omega_1$ starting from $x$ travelling in the $\xi$ direction at constant speed and in straight line intersects $\d \Omega \cup \d \Omega_2$ and $\xi^1 \in \R^2$ is the direction of the outgoing ray reflected according to the law of the optic geometry. We highlight that not all $\xi \in \R^2$ are admissible directions for \eqref{billardmap} but it is costumary to identify these directions to their unique outgoing direction. We further assume that the boundary $\d \Omega_2 \setminus \Gamma_2$ is parametrized by $\delta_2(s), s\in [0,1]$. 

\begin{definition}[Uniformly escaping geometry]
We say that $\Omega_1^f$ is a uniformly escaping geometry if the application 
\begin{align*}\label{directionmap}
\M:  (\d \Omega_2 \setminus \Gamma_2) \times \R^2 & \longrightarrow \qquad \qquad \R \\
(x,\xi) \quad & \longmapsto  \<\xi,n\left(\Pi_x\left(\F\left(x,\xi \right)\right)\right)^\perp\>
\end{align*}
is nondecreasing for $s \mapsto \M(\delta_2(s),n_2(\delta_2(s))), \delta_2(s)\in \d \Omega_2 \setminus \Gamma_2$.
\end{definition}

The name escaping geometry refers to the work of Miller in \cite{Escape} on escape functions where GCC is reinterpreted in terms of escaping rays. This notion is appropriate in the context of rays crossing an interface as there is two ways for rays to escape $\Omega_1^f$ : either through the interface $\d \Omega_2 \setminus \Gamma_2$ or by $\Omega_1 \setminus \Omega_1^f$.

The uniformly escaping geometry ensures that every rays propagating in $\Omega_1^f$ will be observed. Indeed, by definition, a ray starting in the $\M(x,n_2(x))=0, x\in \d \Omega_2 \setminus \Gamma_2$ region and in the $n_2(x)$  direction satisfy $\F^2(x,n_2(x))=(x,n_2(x))$ by definition. Since this is an escaping direction for $\d \Omega_2 \setminus \Gamma_2$, this ray is assumed to have escaped $\Omega_1^f$ (see the light green ray in figure \ref{Uegfig} on the left). This ray will in fact be observed by $\Gamma_2$. A ray starting in the $\M(x,n_2(x))<0$ region in the $n_2(x)$ direction will eventually escape through $\Omega_1 \setminus \Omega_1^f$ thanks to the nondecreasing assumption on $\M$ (see the light green ray in figure \ref{Uegfig} on the right). The same description holds for rays in the $\M(x,n_2(x))>0$ region in the $n_2(x)$ direction. The complete picture can be deduced by this analysis. Indeed, consider a ray in the $M(x,n_2(x))<0$ region propagating in the opposite direction of the parametrization of $\delta$. It is always possible to follow such a ray since $\xi \neq n_2(x)$ implies that $x$ is a point where a reflection occurs. One can then choose to follow the half-ray propagating in the opposite direction of the parametrization. Since the ray propagating in the $n_2(x)$ direction have escaped in the opposite direction of propagation, then the half-ray propagating in the same direction also escape through the same boundary (see the dark green ray in figure \ref{Uegfig} on the right). The case $M(x,n_2(x))<0$ is symmetric and one can follow either half-ray in the region $M(x,n_2(x))=0$ as both half-ray escapes uniformly (see the dark green ray in figure \ref{Uegfig} on the left for the propagation of one of the half-ray). 

\begin{figure}[!ht]
\begin{center}
	\includegraphics[height=2cm]{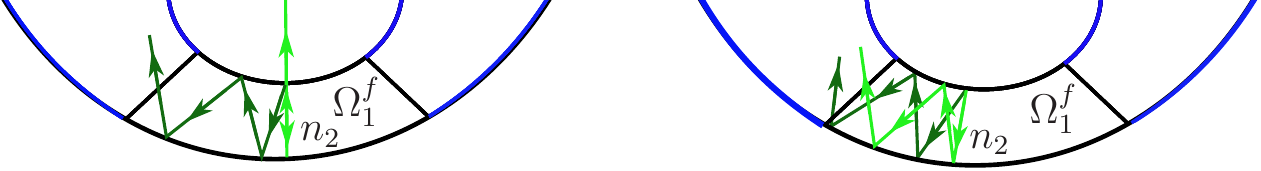} 
\end{center}
\caption{\footnotesize{Example of a uniformly escaping geometry. Left : (light green) ray from $x$ in the $n_2$ direction such that $\M(x,n_2(x))=0$ - (dark green) half-ray from the same point propagating in the negative tangential direction (with respect to $\delta'$). Right : (light green) ray from $x$ in the $n_2$ direction such that $\M(x,n_2(x))<0$ - (dark green) half-ray from the same point propagating in the negative tangential direction (with respect to $\delta'$). }}\label{Uegfig}
\end{figure}

The weak observability then follow from the uniformly escaping geometry condition. 

\begin{lemma}
Suppose $\Omega_1^f$ satisfies the uniformly escaping geometry condition. Then, $\mu, \nu=0$ near $\rho \in T^*(\Omega_1^f \times (0,T))$. 
\end{lemma}

\beginpf

From the uniformly escaping geometry condition and by the geometrical construction, every ray escaping to $\Omega_1 \setminus \Omega_1^f$ is observed. Therefore, consider first $(\rho_1,\rho_2) \in \H^1 \times \E^2$ such that $\Pi_x(\rho_1)\in \d \Omega_2 \setminus \Gamma_2$. Then, from the uniformly escaping geometry condition implies that one of the two outgoing half-ray escapes to $\Omega_1 \setminus \Omega_1^f$ after, eventually, a finite number of reflection on $\Gamma_2$. If there is no reflection on $ \d \Omega_2 \setminus \Gamma_2$ this implies that $\mu, \nu=0$ near $(\rho_1,\rho_2)$. The same holds if there is no transmission. Otherwise, if there is a transmission, then it suffices to follow the transmitted half ray propagating in $\Omega_2$ (locally) in the direction of the escaping ray. This ray may eventually intersect $\Gamma_2$ a finite number of time before reaching $\Gamma_2$ (that satisfies GCC for $\Omega_2$). For every intersection with $\Gamma_2$, one uses the same argument to conclude that this ray is also observed (see figure \ref{figpropueg}). The previous argument holds for $(\rho_1,\rho_2) \in \H^1 \times (\H^2 \cup \G^2)$ which ends the proof. 

\endpf

\begin{figure}[!ht]
\begin{center}
	\includegraphics[height=2cm]{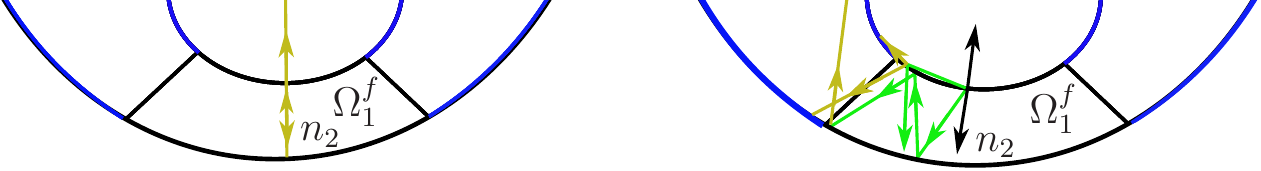}
\end{center}
\caption{\footnotesize{Left : the ray starting from $x$ in the $n_2$ direction such that $\M(x,n_2(x))=0$ is directly observed when transmitted to $\Omega_2$. Right : the propagation of the half-rays in the negative tangential direction from a point $x$ such that $\M(x,n_2(x))<0$.}}\label{figpropueg}
\end{figure}

\subsection{Strong observability}

We finally use the weak observability inequality to prove the observability inequality. To this end, we define the set of invisible solutions 
\begin{align*}
N_T:= & \Bigg\{ (u,v,\eta) \in H^1((0,T)\times \Omega_1) \times H^1((0,T)\times \Omega_2) \times H^1((0,T);H^1(\R^+;V)) \, \Bigg|   \\
& \, \, \, \, \,  (u,v,\eta) \textrm{ is a solution of } \eqref{AUP} \textrm{ s.t. } \restriction{(u,u_t,v,v_t,\eta)}{t=0}\in X, \restriction{\eta}{s=0}=0,    \\
& \, \, \, \,   \textrm{ and} \int_0^\infty (-g'(s)) \int_{\Omega_1} b(x) | \nabla \eta |^2 dxds =0  \Bigg\}
\end{align*}
endowed with the norm $\| (u,v,\eta) \|_{N_T} = \| (u,v) \|_{H_1}+\|\eta\|_{H^1((0,T);H^1(\R^+;V))}$. 
\begin{lemma}\label{leminv}
We have $N_T=\{(0,0,0)\}$. 
\end{lemma}
We follow closely the classical proof (see for instance \cite{Peppino}).

\beginpf

The set $N_T$ is closed by definition and the uniform escaping geometry assumption allows us to conclude that the solutions of $N_T$ are smooth. Notice that \eqref{AUP} is time-invariant. Therefore, if $(u,v,\eta)$ is a smooth solution of \eqref{AUP}, so is $\d_t (u,v,\eta)$. Therefore, if $(u,v,\eta) \in N_T$, then $\d_t (u,v,\eta) \in N_T$. Moreover, from the weak observability and the compact embedding from $X$ to $X^{-1}$, we conclude that $N_T$ is finite-dimensional. 

We now prove \ref{leminv} by contradiction. We begin by noticing, similarly to Lemma \ref{convetaen}, that 
\[
\int_0^\infty (-g'(s)) \int_{\Omega_1} b(x) | \nabla \eta |^2 dxds =0, 
\]
forces $\eta=0$ over $\textrm{supp}(b)$. So the proof boils down to the proof similar to that of the classical wave equation, which we recall. Assume $(u,v,\eta) \in N_T$ and $(u,v) \neq (0,0)$. From the transmission conditions at the interface, we deduce that if $u=0$, then $v=0$ since every bicharacteristics in $\Omega_2$ encounters $\d \Omega_2$. Therefore we consider $u \neq 0$. Since $N_T$ is finite-dimensional, $\partial_t : N_T \rightarrow N_T $ has at least one complex eigenvalue $\lambda$ such that $\partial_t (u,v) = \lambda (u,v)$. Therefore $(u,v)$ is of the form $(u,v)=e^{\lambda t} (U(x),V(x))$. But from $\eta=0$ over $\textrm{supp}(b)$, we use $\eta(x,t,s)=u(x,t)-u(x,t-s)$, and using the expression of $u$ and $\eta$, we have
\[
(e^t-e^{t-s})U(x)=0, \quad  \textrm{in }L^2((0,T) \times \R^+; H^1(\textrm{supp}(b)).
\]
This implies that $U=0$ in $H^1( \textrm{supp}(b))$ and the unique continuation properties of $(\lambda^2 - k_1 \textrm{div}(1-k_0b(x) \nabla )) $ allows us to conclude that $U=0$ in $H^1( \Omega_1)$ and therefore $V=0$ in $H^1( \Omega_2)$ and, by definition, $\eta=0$ in $H^1(\Omega_1 \times (0,\infty) \times (0,T))$, which is a contradiction and ends the proof. 

\endpf

We finally gather everything to conclude on the proof of Theorem \ref{main}. 

\beginpf

Under the analytical and geometrical assumptions, we conclude on the observability of \eqref{AUP}. Since this equation is autonomous, then the observability implies the exponential stability of \eqref{AUP}, which implies the exponential stability of \eqref{MP1}

\endpf

\bibliographystyle{plain}
\bibliography{Viscoelastic}

\end{document}